\documentclass{amsart}
\usepackage{a4}
\usepackage{amssymb}
\usepackage{amsmath}
\usepackage{amsthm}
\usepackage{enumerate}
\usepackage{mathrsfs}  

\usepackage{stmaryrd}

\def\squareforqed{\hbox{\rlap{$\sqcap$}$\sqcup$}}
\def\qed{\ifmmode\squareforqed\else{\unskip\nobreak\hfil
\penalty50\hskip1em\null\nobreak\hfil\squareforqed
\parfillskip=0pt\finalhyphendemerits=0\endgraf}\fi\medskip}

\def\ovA{\mathord{\mbox{\hspace*{0pt}}\,\overline{\!A\!\!\;}\!\:}\mbox{\hspace*{0pt}}}
\def\ovB{\mathord{\mbox{\hspace*{0pt}}\,\overline{\!B\!\!\;}\!\:}\mbox{\hspace*{0pt}}}
\def\ovC{\mathord{\mbox{\hspace*{0pt}}\,\overline{\!C\!\!\;}\!\:}\mbox{\hspace*{0pt}}}
\def\ovD{\mathord{\mbox{\hspace*{0pt}}\,\overline{\!D\!\!\;}\!\:}\mbox{\hspace*{0pt}}}
\def\ovE{\mathord{\mbox{\hspace*{0pt}}\,\overline{\!E\!\!\;}\!\:}\mbox{\hspace*{0pt}}}
\def\ovF{\mathord{\mbox{\hspace*{0pt}}\,\overline{\!F\!\!\;}\!\:}\mbox{\hspace*{0pt}}}
\def\ovM{\mathord{\mbox{\hspace*{0pt}}\,\overline{\!M\!\!\;}\!\:}\mbox{\hspace*{0pt}}}

\def\ovV{\mathord{\mbox{\hspace*{0pt}}\,\overline{\!V\!\!\;}\!\:}\mbox{\hspace*{0pt}}}

\def\ovt{\mathord{\mbox{\hspace*{0pt}}\:\!\overline{\!\!\;t}}\mbox{\hspace*{0pt}}}
\def\ovu{\mathord{\mbox{\hspace*{0pt}}\,\overline{\!u\!\!\;}\!\:}\mbox{\hspace*{0pt}}}
\def\ovv{\mathord{\mbox{\hspace*{0pt}}\,\overline{\!v\!\!\;}\!\:}\mbox{\hspace*{0pt}}}

\def\ovx{\mathord{\mbox{\hspace*{0pt}}\,\overline{\!x\!\!\;}\!\:}\mbox{\hspace*{0pt}}}
\def\ovy{\mathord{\mbox{\hspace*{0pt}}\,\overline{\!y\!\!\;}\!\:}\mbox{\hspace*{0pt}}}
\def\ovz{\mathord{\mbox{\hspace*{0pt}}\,\overline{\!z\!\!\;}\!\:}\mbox{\hspace*{0pt}}}

\def\ombb{\mathord{\mbox{\hspace*{0pt}}\;\!\overline{\!\!\:\omega\!\!\:}\;\!}\mbox{\hspace*{0pt}}}

\def\cn{\mathord{{\;\!{:}\;\!}}}

\def\NN{\mathord{\mathrm{N}}}
\def\II{\mathord{\mathrm{I}}}

\def\OO{\mathord{\mathbb{O}}}

\newcommand{\Fq}{\mathbb{F}_q}
\newcommand{\Fqs}{\mathbb{F}_{q^2}}

\newcommand{\omg}{\overline{\omega}}

\newcommand{\diag}{\mathrm{diag}}
\newcommand{\PSL}{\mathrm{PSL}}
\newcommand{\GL}{\mathrm{GL}}
\newcommand{\SL}{\mathrm{SL}}
\newcommand{\GO}{\mathrm{GO}}

\newcommand{\SO}{\mathrm{SO}}
\newcommand{\SU}{\mathrm{SU}}

\newcommand{\SE}{\mathrm{SE}}
\newcommand{\E}{\mathrm{E}}

\newcommand{\Spin}{\mathrm{Spin}}
\newcommand{\Tr}{\operatorname{T}}

\newcommand{\J}{\mathbb{J}}
\newcommand{\rad}{\mathrm{rad}}
\renewcommand{\bar}{\overline}
\renewcommand{\hat}{\widehat}

\newcommand{\B}{\mathrm{B}}
\renewcommand{\H}{\mathrm{H}}
\newcommand{\s}{\sigma}
\renewcommand{\ker}{\operatorname{ker}}
\newcommand{\T}{\top}

\newcommand{\fdet}{\Delta}

\newcommand{\inner}[2]{\langle #1, #2 \rangle}

\makeatletter
\renewcommand*\env@matrix[1][c]{\hskip -\arraycolsep
  \let\@ifnextchar\new@ifnextchar
  \array{*\c@MaxMatrixCols #1}}
\makeatother

\usepackage{marginnote}
\usepackage{marginfix}

\theoremstyle{plain}
\newtheorem{theorem}{Theorem}
\numberwithin{theorem}{section}
\newtheorem{lemma}[theorem]{Lemma}

\newtheorem{proposition}[theorem]{Proposition}

\theoremstyle{definition}
\newtheorem{definition}[theorem]{Definition}

\title
[Octonions, Albert vectors and ${}^2\E_6(F)$]
{Octonions, Albert vectors and the group ${}^2\E_6(F)$}
\author{John~N.~Bray,\ Yegor~Stepanov,\ Robert~A.~Wilson}
\address{}
\email{johnnbray@hotmail.com, yegor@ystepanoff.net, r.a.wilson@qmul.ac.uk}

\begin{document}

\begin{abstract}
We extend the explicit octonionic approach to groups of type $\E_6$ to
the twisted groups ${}^2\SE_{6,K}(F)$ associated with a quadratic Galois
extension $K/F$. These groups act on the Albert space over $K$,
preserving the Dickson--Freudenthal determinant and a Hermitean form.
When the field norm $K^{\times}\to F^{\times}$ is surjective, we give
explicit generators and co\"ordinate proofs of the three-orbit theorem
on white points and of the stabilisers of representative white vectors.
Under this hypothesis we also prove simplicity of the central quotient.
Over finite fields we recover the classical orbit lengths and group orders from the
stabiliser orders and the total number of white points. Finally, without
assuming norm surjectivity, we show that the two isotropic white-point
orbits remain unchanged, while the non-isotropic orbits are parametrised
by $F^{\times}/\NN_{K/F}(K^{\times})$.
\end{abstract}

\maketitle

\tableofcontents

\section{Introduction}

In \cite{BSW1} we gave an explicit construction of the groups of type
$\E_6$ over arbitrary fields using split octonions. The group acts on
the $27$-dimensional Albert space, whose vectors are written as
$(a,b,c\mid A,B,C)$, with three scalar and three octonion co\"ordinates,
and preserves the Dickson--Freudenthal determinant. Suitable
$3\times3$ octonionic matrices encode linear transformations of this
space and permit direct calculations with group elements. The purpose
of the present paper is to extend this approach to the twisted groups
${}^2\SE_{6,K}(F)$ associated with a quadratic Galois extension $K/F$.
We retain the Albert space and its determinant, and require the
transformations to preserve an additional Hermitean form.

The description of exceptional groups by forms on their smallest
modules is well established. Aschbacher \cite{Asch1} studies the
$27$-dimensional module and its geometry, including the twisted group;
his treatment of spaces with a quadratic form and a Hermitean form in
\cite[Section~10]{Asch1} is used here. The relation between octonions,
Albert algebras and exceptional groups is developed more generally in
\cite{SprVeld}. In the finite case, Wilson
\cite[Sections~9--10]{WilsonAlbert} describes constructions of
${}^2\SE_6(q)$ in the Albert space and gives the three white-point orbit
lengths and the group order. The orbit lengths are stated in
\cite[Section~10]{WilsonAlbert}, but the corresponding calculations
are left implicit. One purpose of the present paper is to supply these
calculations. We determine the orbits and their stabilisers explicitly,
and then derive the orbit lengths and group order from the total
number of white points. Thus the formulae themselves are already
known; the contribution here is a detailed proof in octonion
co\"ordinates. Simplicity is likewise a classical result, for which we
give a proof within the present construction.

The treatment also extends to quadratic Galois extensions under the
hypotheses stated below, and makes explicit the role of the field norm
in the orbit classification.
We use the determinant-preserving actions and white-point geometry
from \cite{BSW1}, rather than reconstructing the split group. The
additional Hermitean form requires different families of elementary
transformations and a separate analysis of the stabilisers. In
particular, the single orbit of the split group on white points breaks
into several orbits for the twisted group. We obtain representatives
and stabilisers by calculations in the Albert space, using standard
facts about classical groups. The calculations describe the normal
subgroups and their complements explicitly, rather than only giving
the abstract stabiliser types. No restriction on the characteristic
is imposed.

Throughout the paper, $K/F$ is a quadratic Galois extension and $\sigma$ is its
non-trivial automorphism. Except in Section~\ref{section:conclusions}, we impose
the additional standing hypothesis that the
field norm is surjective:
\begin{equation}
\label{eq:1_norm_surjective}
\tag{N}
\NN_{K/F}(K^{\times})=F^{\times},\qquad
\NN_{K/F}(t)=tt^{\s}.
\end{equation}
Thus every non-zero element of $F$ has the form $tt^{\s}$ for some
$t\in K^{\times}$. This holds, in particular, when $F$ is finite. We make the
uses of this hypothesis explicit in the non-isotropic orbit arguments.

The hypothesis is important because replacing an Albert vector $X$ by $tX$
multiplies its Hermitean value by $tt^{\s}$:
\[
\H(tX)=tt^{\s}\H(X).
\]
Consequently, if $\H(X)\neq0$, its class in
$F^{\times}/\NN_{K/F}(K^{\times})$ is an invariant of the orbit of the point
$\langle X\rangle$. Hypothesis~\eqref{eq:1_norm_surjective} removes this
obstruction by allowing every non-zero Hermitean value to be normalised to $1$.
Extensions for which the norm map is not surjective are treated in
Section~\ref{section:conclusions};
the three-orbit assertion below is made under hypothesis~\eqref{eq:1_norm_surjective}.

Under this hypothesis, Theorem~\ref{theorem:6_three_orbits} gives two
orbits of isotropic white points and one orbit of non-isotropic white
points. Here a point means a $1$-dimensional $K$-subspace, and isotropy
refers to the Hermitean form. The two isotropic types are distinguished
by the radical of this form on the $17$-space attached to a white
point. The stabilisers of chosen vector representatives of the three
types have shapes
\[
    (F^8.K^8)\cn\Spin_{8,K}^{-}(F),\qquad
    (F.K^{10})\cn\SU_{5,K}(F),\qquad
    \Spin_{10,K}^{-}(F),
\]
respectively; see Theorems~\ref{theorem:7_type1_stabiliser},
\ref{theorem:7_type2_stabiliser} and
\ref{theorem:7_type3_stabiliser}. Powers of $F$ and $K$ here denote
additive groups. The extensions denoted by dots in the first two
stabilisers are non-split, whereas the displayed outer extensions are
semidirect products. These are vector stabilisers; the extra scalar
multipliers needed for point stabilisers are determined separately
in the finite case.

The same co\"ordinate description gives a proof of simplicity of
${}^2\E_{6,K}(F)$ by Iwasawa's lemma. When $F=\Fq$ and $K=\Fqs$,
we combine the three point-stabiliser orders with the total number
of white points from \cite{BSW1}. The orbit--stabiliser formula then
determines all three orbit lengths and the order of ${}^2\SE_6(q)$
simultaneously. Factoring out the scalar centre gives the usual order
of ${}^2\E_6(q)$. Thus these familiar formulae are consequences of the
explicit stabiliser calculations, not inputs to them.

For arbitrary quadratic Galois extensions, the orbit statement has
a further field-dependent part. In
Theorem~\ref{theorem:conclusions_norm_orbits} we prove that there are
still exactly two isotropic white-point orbits, and that the
non-isotropic orbits correspond precisely to the elements of
$F^{\times}/\NN_{K/F}(K^{\times})$. This gives the exact condition
under which the finite-field three-orbit statement remains valid.
The stabiliser and simplicity theorems in the main body
retain hypothesis~\eqref{eq:1_norm_surjective}.

The paper is organised as follows. Sections~2 and~3 recall the
octonionic model and the required facts about spaces with two forms.
Sections~4 and~5 define the Hermitean form and construct the group
elements used subsequently. Sections~6 and~7 determine the white-point
orbits and the vector stabilisers. Section~\ref{section:simplicity}
proves simplicity, Section~9 treats finite fields, and
Section~\ref{section:conclusions} discusses the orbit classification
when~\eqref{eq:1_norm_surjective} is omitted.

\section{Split octonions and the Albert space $\J$}

We briefly recall the definitions from \cite{BSW1}. The split octonion algebra $\OO = \OO_F$
over a field $F$ is an $8$-dimensional vector space with basis $\{e_i \mid i \in \pm I\}$,
where \mbox{$\pm I = \{\pm 0, \pm 1, \pm \omega, \pm \ombb\}$}.
The multiplication is determined by the rules
\begin{enumerate}[(i)]
	\item $e_1 e_{\omega} = -e_{\omega} e_1 = e_{-\ombb}$,
	\item $e_1 e_0 = e_{-0} e_1 = e_1$,
	\item $e_{-1} e_1 = -e_0$ and $e_0 e_0 = e_0$,
\end{enumerate}
together with their images under negating all subscripts and multiplying all subscripts
by $\omega$ (where $\omega^2 = \ombb$ and $\omega \ombb = 1$). All other products of
basis vectors are $0$.
The elements $e_0$ and $e_{-0}$ are orthogonal idempotents with $e_0 + e_{-0} = 1_{\OO}$.
For $x = \sum_{i \in \pm I} \lambda_i e_i$, the trace and norm are
\begin{equation}
	\Tr(x) = \lambda_0 + \lambda_{-0}, \quad
	\NN(x) = \lambda_{-1} \lambda_1 + \lambda_{\omg} \lambda_{-\omg} +
		\lambda_{\omega}\lambda_{-\omega} + \lambda_0 \lambda_{-0},
\end{equation}
and conjugation is defined by $e_i \mapsto -e_i$ for $i \neq \pm 0$ and $e_0 \leftrightarrow e_{-0}$.
We say that an $F$-subalgebra $S$ of $\OO$ is \textit{sociable} if $S$ contains $F \cdot 1_{\OO}$ and
for all $x,y \in S$ and for all $z \in \OO$ we have $(xy)z = x(yz)$.

The Albert space $\J = \J_F$ is the $27$-dimensional vector space of elements
\begin{equation}
	X = (a,b,c\mid A,B,C) = \begin{bmatrix}
		a & C & \ovB \\
		\ovC & b & A \\
		B & \ovA & c
	\end{bmatrix},
\end{equation}
where $a,b,c \in F$ and $A,B,C \in \OO$. The Dickson--Freudenthal determinant is the cubic form
\begin{equation}
	\fdet(X) = abc - aA\ovA - bB\ovB - cC\ovC + \Tr(ABC),
\end{equation}
and we define the group $\SE_6(F)$ to be the group of all $F$-linear maps on $\J$
preserving $\fdet$.  We use exponent notation for right actions:
$X^g$ denotes the image of $X$ under $g$, so that
\begin{equation*}
	 (X^g)^h=X^{gh}.
\end{equation*}
For a subspace $U$, we similarly write
$U^g=\{X^g\mid X\in U\}$.  As in
\cite{BSW1}, we encode some elements of $\SE_6(F)$ by $3\times 3$
matrices $M$ written over sociable subalgebras of $\OO$.  For such a
matrix,
\begin{equation*}
	 X^M=\ovM^{\T} X M.
\end{equation*}
This is well-defined since every entry of $\ovM^{\T} X M$ is a sum of
terms $m_1 x m_2$ with $m_1,m_2$ in the same sociable subalgebra.

Suppose $X = (a,b,c \mid A,B,C)$ and $Y = (d,e,f\mid D,E,F)$. The mixed form $M(Y,X)$ is defined as
\begin{multline}
	M(Y,X) = bcd + ace + abf - d A\ovA  -e B \ovB - f C\ovC \\
	- a(D\ovA + A\ovD) - b(E\ovB + B\ovE) - c(F\ovC + C\ovF) + \Tr(DBC + ECA + FAB).
\end{multline}
We colour the non-zero Albert vectors in $\J$ according to the following rules.

\begin{definition}
	A non-zero Albert vector $X \in \J$ is called
	\begin{enumerate}[(i)]
		\item \textit{white} if $M(Y,X) = 0$ for all $Y \in \J$;
		\item \textit{grey} if $\fdet(X) = 0$ and there exists $Y \in \J$ such that
			$M(Y,X) \neq 0$;
		\item \textit{black} if $\fdet(X) \neq 0$.
	\end{enumerate}
	A \textit{white}, \textit{grey}, or \textit{black point} is a $1$-dimensional subspace
	of $\J$ spanned by a white, grey, or black vector, respectively.
\end{definition}

The following criterion for whiteness will be of particular importance.

\begin{lemma}
	\label{lemma:2_whiteness}
	An Albert vector $X = (a,b,c\mid A,B,C)$ is white if and only if the following
	conditions hold:
	\begin{equation}
		\label{eq:2_whiteness}
		\left.\begin{array}{r@{\;}c@{\;}l}
			A \ovA & = & bc, \\
			B \ovB & = & ca, \\
			C \ovC & = & ab, \\
			AB & = & c\ovC, \\
			BC & = & a\ovA, \\
			CA & = & b\ovB.
		\end{array}
		\right\}.
	\end{equation}
	If $X$ is white, then $\fdet(X) = 0$.
\end{lemma}

The group $\SE_6(F)$ acts transitively on white points and preserves the colour of
Albert vectors.

\section{Spaces with two forms}

In this section we describe the results from \cite[Section 10]{Asch1}, which will be of great importance to us in
the further discussion.

Let $K/F$ and $\s$ be as in the Introduction, with the norm-surjectivity
hypothesis~\eqref{eq:1_norm_surjective}. In particular, $-1$ is a field norm,
as required in \cite[Section~10]{Asch1}.
Let also $V$ be a $2m$-dimensional vector space over $K$ with a quadratic form $Q$ and a conjugate-symmetric
sesquilinear form \mbox{$B : V\times V \rightarrow K$}, defined with respect to $\s$. Denote by $f$ the polar
form of $Q$, and suppose that  \mbox{$\mathcal{B}$} is a basis of $V$ with respect to $f$, consisting of the
elements $e_1, \ldots, e_m$, $f_1,\ldots,f_m$, so that
\begin{equation}
	Q(e_i) = Q(f_i) = 0,\ f(e_i, f_i) = 1,
\end{equation}
and $f(e_i,e_j) = f(f_i,f_j) = 0 = f(e_i,f_j)$ for $i \neq j$. The above said means that $(V,Q)$ is a hyperbolic
orthogonal space. Denote by $G$ the maximal amongst all the subgroups of $\GO(V, Q)$ which preserve $B$. If $U$
is a subspace of $V$, then we denote the restrictions of $f$ and $B$ on $U$ as $f_U$ and $B_U$ respectively. We 
say that an element $v \in V$ is \textit{singular isotropic}, if $Q(v) = B(v,v) = 0$. 

\begin{definition}
	\label{def:qb_subspace}
	A $(Q,B)$-subspace of $V$ is an $F$-subspace $U$ of $V$ such that the following holds: \mbox{$V = U \otimes_F K$}, 
	$f_U = B_U$ is an $F$-form on $U$, and $Q_U$ is a non-degenerate quadratic form on $U$ of 
	Witt index at least $2$. 
\end{definition}

\begin{proposition}[(10.1) in \cite{Asch1}]
	\label{prop:3_2forms}
	If $U$ is a $(Q,B)$-subspace of $V$, then it is the unique $(Q,B)$-subspace of $V$, and
	$G = \GO(U,Q)$.
\end{proposition}

\begin{proof}
	The subspace is determined by the two forms through
	\[
		U=\{v\in V:B(w,v)=f(w,v)\text{ for every }w\in V\}.
	\]
	Indeed, the equality holds for $v\in U$, first for $w\in U$ and then,
	by $K$-linearity in $w$, for all $w\in V$. Conversely, fix
	$\lambda\in K\setminus F$ and write $v=u_1+\lambda u_2$ with $u_1,u_2\in U$.
	For $w\in U$ the equality gives
	$(\lambda^{\s}-\lambda)f(w,u_2)=0$, so non-degeneracy forces $u_2=0$.
	This also proves uniqueness.

	Every element of $G$ therefore preserves $U$ and restricts to an isometry
	of $Q_U$. Conversely, an isometry of $Q_U$ extends $K$-linearly to $V$
	preserving $Q$ and $B$, since $B_U=f_U$. Hence $G=\GO(U,Q)$.
\end{proof}

\begin{proposition}[(10.5) in \cite{Asch1}]
	\label{prop:3_2forms_orbits}
	Let $U$ be a $(Q,B)$-subspace of $V$ of Witt index at least $2$. Fix
	\mbox{$\lambda \in K \setminus F$}. The group $G$ has
	precisely two orbits on singular isotropic points with representatives
	$\langle u \rangle$ and \mbox{$\langle u + \lambda v \rangle$}, where $u,v \in U$ and
	\mbox{$\langle u, v \rangle$} is a singular line.
\end{proposition}

\begin{proof}
	The two displayed points are singular isotropic because
	$Q(u)=Q(v)=f(u,v)=0$ and $B=f$ on $U$.
	Since $G$ preserves $U$, points having a non-zero representative in $U$
	cannot be conjugate to those having none. Witt's theorem gives
	transitivity on the former.

	For a point of the latter kind, write a representative as
	$x=u_1+\lambda u_2$, where $u_1,u_2\in U$ are $F$-independent.
	Comparing coefficients in $Q(x)=0$ gives
	\[
		Q(u_1)=\lambda\lambda^{\s}Q(u_2),\qquad
		f(u_1,u_2)=-(\lambda+\lambda^{\s})Q(u_2).
	\]
	Consequently,
	\[
		0=B(x,x)=-(\lambda-\lambda^{\s})^2Q(u_2).
	\]
	Thus $Q(u_1)=Q(u_2)=f(u_1,u_2)=0$.
	The map $u_1\mapsto u$, $u_2\mapsto v$ is therefore an isometry between
	totally singular planes. By Witt's theorem it extends to an isometry
	of $U$, whose $K$-linear extension belongs to $G$ and sends $x$ to
	$u+\lambda v$.
\end{proof}

\begin{lemma}
	\label{lemma:3_2forms_nonisotropic}
	Assume that $U$ is a $(Q,B)$-subspace of $V$ and that the field norm
	$\NN_{K/F}:K^{\times}\rightarrow F^{\times}$ is surjective. Then $G$ is
	transitive on the points $\langle x\rangle$ with $Q(x)=0$ and $B(x,x)\neq0$.
\end{lemma}

\begin{proof}
	Let $x,y\in V$ satisfy $Q(x)=Q(y)=0$ and $B(x,x),B(y,y)\neq0$.
	Here we use the surjectivity of the field norm: choose $t\in K^{\times}$ with
	\[
		tt^{\s}=\frac{B(x,x)}{B(y,y)}.
	\]
	Replacing $y$ by $ty$ does not change its point, and gives $B(y,y)=B(x,x)$.

	Fix $\lambda\in K\setminus F$, put $d=\lambda-\lambda^{\s}\neq0$, and write
	\[
		x=u_1+\lambda u_2,\qquad y=w_1+\lambda w_2,
		\qquad u_1,u_2,w_1,w_2\in U.
	\]
	Using $\lambda^2=(\lambda+\lambda^{\s})\lambda-\lambda\lambda^{\s}$
	and comparing coefficients in $Q(x)=0$, we obtain
	\[
		Q(u_1)=\lambda\lambda^{\s}Q(u_2),\qquad
		f(u_1,u_2)=-(\lambda+\lambda^{\s})Q(u_2).
	\]
	Since $B=f$ on $U$, it follows that
	\[
		B(x,x)=2Q(u_1)+(\lambda+\lambda^{\s})f(u_1,u_2)
		       +2\lambda\lambda^{\s}Q(u_2)=-d^2Q(u_2).
	\]
	The same identities hold for $y$ and $w_1,w_2$. Hence $B(x,x)=B(y,y)$
	implies
	\[
		Q(u_1)=Q(w_1),\qquad Q(u_2)=Q(w_2),\qquad
		f(u_1,u_2)=f(w_1,w_2).
	\]
	Moreover, the determinant of the polar Gram matrix on $\langle u_1,u_2\rangle_F$
	is $-d^2Q(u_2)^2\neq0$, and likewise for $\langle w_1,w_2\rangle_F$.
	Thus the map $u_i\mapsto w_i$ is an isometry between these two
	non-degenerate planes. By Witt's theorem it extends to an isometry of $U$.
	Its $K$-linear extension preserves $Q$ and also $B$, since $B=f$ on $U$.
	It therefore belongs to $G$ and sends $x$ to $y$, proving the assertion.
\end{proof}

\section{Hermitean form in $\J$ and the group ${}^2\SE_{6,K}(F)$}

Retain the quadratic Galois extension $K/F$, its non-trivial automorphism
$\sigma$, and hypothesis~\eqref{eq:1_norm_surjective}. Let $\OO_F$ be
a split octonion algebra over $F$ and $\OO_K = \OO_F \otimes_F K$.

We slightly abuse the notation here denoting by $\s$ also the automorphism of $\OO_K$ induced by the field automorphism $\sigma$:
\begin{equation}
	\left( \sum_{i\in \pm I} \lambda_i e_i \right)^{\s} = 
	\sum_{i\in \pm I} \lambda_{i}^{\sigma} e_i. 
\end{equation}
This, however, should not create any difficulties for the reader. Consider the following Hermitean form defined on the elements
of $\OO_K$:
\begin{equation}
	h(x) = x \ovx^{\s} + x^{\s} \ovx = \Tr(x \ovx^{\s}).
\end{equation}
On the Albert space $\J = \J_K$ this induces the Hermitean form $\H$, where
\begin{equation}
	\H( (a,b,c\mid A,B,C) ) = 
		a a^{\s} + b b^{\s} + c c^{\s} + 
		\Tr(A \ovA^{\s} + B\ovB^{\s} + C\ovC^{\s}).
\end{equation}
We will use both the quadratic map $\H$ and the underlying $\s$-Hermitean sesquilinear form.
For \mbox{$X=(a,b,c\mid A,B,C)$} and \mbox{$Y=(d,e,f\mid D,E,F)$} define
\begin{equation}
	\label{eq:4_B_sesquilinear}
	\B_{\H}(X,Y)= a d^{\s}+b e^{\s}+c f^{\s}+
		\Tr\!\left(A\ovD^{\s}+B\ovE^{\s}+C\ovF^{\s}\right).
\end{equation}
Then $\B_{\H}$ is a conjugate-symmetric $\s$-sesquilinear form on $\J$, and $\B_{\H}(X,X)=\H(X)$ for all $X$.
The polarisation of the quadratic map $\H$ is the $F$-bilinear symmetric form
\begin{equation}
	\inner{X}{Y}_{\H}:=\H(X+Y)-\H(X)-\H(Y)=\B_{\H}(X,Y)+\B_{\H}(Y,X),
\end{equation}
which in general is \emph{not} $\s$-sesquilinear.

Using the construction from \cite{BSW1}, we obtain the group $\SE_6(K)$ as the group of $K$-linear maps on $\J_K$ 
preserving the Dickson--Freudenthal determinant. We define the group ${}^2\SE_{6,K}(F)$ as the subgroup of $\SE_6(K)$ 
which preserves $\H$. In case $F = \Fq$ and $K = \Fqs$, we denote this by ${}^2\SE_6(q)$. The group ${}^2\E_{6,K}(F)$ 
is defined as the quotient of ${}^2\SE_{6,K}(F)$ by its centre.

\section{Some elements of ${}^2\SE_{6,K}(F)$}

Let $X = (a,b,c\mid A,B,C)$ be an arbitrary element of $\J = \J_K$. Recall from \cite{BSW1} that the elements
\begin{equation}
	\delta = \begin{bmatrix}
		0 & 1 & 0 \\
		1 & 0 & 0 \\
		0 & 0 & 1
	\end{bmatrix},\quad
	\tau = \begin{bmatrix}
		0 & 1 & 0 \\
		0 & 0 & 1 \\
		1 & 0 & 0
	\end{bmatrix}
\end{equation}
encode elements of $\SE_6(K)$. We notice that $\delta$ and 
$\tau$ also preserve the Hermitean form $\H$, so they encode elements of ${}^2\SE_{6,K}(F)$.

The elements $P_u$ (see \cite[Section 6]{BSW1}) with $u$ written over the small field and such that $\NN(u) = 1$ are also of interest. 

\begin{lemma}
	\label{lemma:5_pu_hermitean}
	The actions on $\J$ of the elements $P_u = \diag(u,\ovu,1)$ such that $u \in \OO_F$ and $\NN(u) = 1$
	preserve $\H$.
\end{lemma}

\begin{proof}
	Recall that the action on $\J$ is given by
	\begin{equation*}
		\begin{array}{r@{\;}c@{\;}l}
			P_u : (a, b, c \mid A, B, C) & \mapsto & (a, b, c \mid u A, B u, \ovu C \ovu),
		\end{array}
	\end{equation*}
	so the individual terms are being mapped in the following way:
	\begin{equation*}
		\begin{array}{r@{\;}c@{\;}l}
			a a^{\s} & \mapsto & a a^{\s}, \\
			b b^{\s} & \mapsto & b b^{\s}, \\
			c c^{\s} & \mapsto & c c^{\s}, \\
			\Tr( A \ovA^{\s} ) & \mapsto & \Tr( (u A) (\ovA^{\s} \ovu) ), \\
			\Tr( B \ovB^{\s} ) & \mapsto & \Tr( (B u) (\ovu \ovB^{\s}) ), \\
			\Tr( C \ovC^{\s} ) & \mapsto & \Tr( (\ovu C \ovu) (u \ovC^{\s} u) ).
		\end{array}
	\end{equation*}
	Note that $u^{\s} = u$ since $u \in \OO_F$. We find
	\begin{multline*}
		\Tr( (u A) (\ovA^{\s} \ovu) ) = \Tr( ((u A) \ovA^{\s}) \ovu ) = 
		\Tr( \ovu ((u A) \ovA^{\s}) ) \\
		= \Tr( ( \ovu (u A) ) \ovA^{\s} ) = 
		\Tr( ( (\ovu u) A ) \ovA^{\s} ) = \Tr(A \ovA^{\s}).
	\end{multline*}
	Similarly,
	\begin{equation*}
		 \Tr( (B u) (\ovu \ovB^{\s}) ) = \Tr( ((B u) \ovu ) \ovB^{\s} ) = 
		 \Tr( (B (u \ovu) ) \ovB^{\s} ) = \Tr(B \ovB^{\s}).
	\end{equation*}
	For the last term we have
	\begin{equation*}
		\Tr( (\ovu C \ovu) (u \ovC^{\s} u) ) = \inner{\ovu C \ovu}{\ovu C^{\s} \ovu}.
	\end{equation*}
	Now, as we know from Lemma 6.2 in \cite{BSW1}, the map $x \mapsto \ovu x \ovu$ is
	a product of two reflexions, hence,  
	\begin{equation*}
		\inner{\ovu C \ovu}{\ovu C^{\s} \ovu} = \inner{C}{C^{\s}} = \Tr(C\ovC^{\s}).
	\end{equation*}
\end{proof}

\begin{lemma}
	\label{lemma:5_isotropic}
	Let $x \in \OO_K$ be such that $\ovx^{\s} x = x \ovx^{\s} = 0$. Then $x\ovx = 0$.
\end{lemma}

\begin{proof}
	If $x = 0$, then the result is trivial. Assume $x$ is non-zero, $\ovx^{\s} x = x\ovx^{\s} = 0$.
	If $x \ovx \neq 0$, then $x$ is invertible, and so $\ovx^{\s} = 0$, which implies $x = 0$, a 
	contradiction.
\end{proof}

Consider the matrices
\begin{equation}
	N_x = \begin{bmatrix}
		1 & x & 0 \\
		-\ovx^{\s} & 1 & 0 \\
		0 & 0 & 1
	\end{bmatrix},\quad
	N_x' = \begin{bmatrix}
		1 & 0 & 0 \\
		0 & 1 & x \\
		0 & -\ovx^{\s} & 1
	\end{bmatrix},\quad
	N_x'' = \begin{bmatrix}
		1 & 0 & -\ovx^{\s} \\
		0 & 1 & 0 \\
		x & 0 & 1
	\end{bmatrix},
\end{equation}
where $x\ovx^{\s} = \ovx^{\s}x = 0$ and $x,\ovx^{\s}$ generate a sociable subalgebra. 
It is easy to see that these encode elements of $\SE_6(K)$. Indeed,
\begin{equation}
	\begin{bmatrix}
		1 & x \\
		0 & 1
	\end{bmatrix}
	\begin{bmatrix}
		1 & 0 \\
		-\ovx ^{\s} & 1
	\end{bmatrix} = 
	\begin{bmatrix}
		1 - x\ovx ^{\s} & x \\
		-\ovx ^{\s} & 1 
	\end{bmatrix} = 
	\begin{bmatrix}
		1 & x \\
		-\ovx ^{\s} & 1
	\end{bmatrix}.
\end{equation}
So, the elements $N_x$, $N_x'$, and $N_x''$ preserve the Dickson--Freudenthal determinant. 
To verify that they preserve $\H$, we look at the action on $\J$:
\begin{equation}
	\label{eq:5_nx_image}
	\begin{array}{r@{\;}c@{\;}l}
		N_x: (a,b,c\mid A,B,C) & \mapsto &
		(a-\Tr(C\ovx ^{\s}), b+\Tr(\ovC x), c \mid \\
		& & \mid	A+\ovx \ovB , B-\ovA \ovx ^{\s},C-x^{\s}\ovC x+ax-bx^{\s}), \\
		
		N_x': (a,b,c\mid A,B,C) & \mapsto & 
		(a, b-\Tr(A \ovx^{\s}), c+\Tr(\ovA x) \mid \\
		& & \mid	A - x^{\s} \ovA x + bx-cx^{\s}, B + \ovx \ovC, C - \ovB \ovx^{\s}), \\
		
		N_x'': (a,b,c\mid A,B,C) & \mapsto &
		(a+\Tr(\ovB x), b, c-\Tr(B\ovx^{\s}) \mid \\
		& & \mid 	A-\ovC\ovx^{\s}, B-x^{\s}\ovB x + cx - ax^{\s}, C+\ovx\ovA).
	\end{array}
\end{equation}
We need to prove an auxiliary lemma.

\begin{lemma}
	\label{lemma:5_oct_aux}
Suppose that $x,y,z \in \OO_K$ with $x \ovx  = 0$. Then
		\begin{enumerate}[(i)]
			\item $x\Tr(yx) = x(yx)$,
			\item $\Tr( (xy)(z\ovx ) ) = 0$.		
		\end{enumerate}
\end{lemma}

\begin{proof}
	\leavevmode
	\begin{enumerate}[(i)]
	
	\item $x\Tr(yx) = x (yx + \ovx \ovy ) = x(yx) + x(\ovx \ovy )
		 = x(yx) + (x\ovx )\ovy = x(yx)$,
		 
	\item $\Tr( (xy)(z\ovx ) ) = \Tr( (z\ovx ) (xy) ) = \Tr( ((z\ovx )x) y ) = 
		\Tr(z(x\ovx )y) = 0$. \qedhere
	\end{enumerate}
\end{proof}

Obviously, it is enough to verify that the elements $N_x$ preserve the Hermitean form $\H$. 
The individual terms in $\H(X) = aa^{\s} + bb^{\s} + cc^{\s} + \Tr(A\ovA ^{\s} + B\ovB ^{\s} +
C\ovC ^{\s})$ are being mapped in the following way:
\begin{equation}
	\begin{array}{r@{\;}c@{\;}l}
		aa^{\s} & \mapsto & aa^{\s} - a^{\s}\Tr(C\ovx ^{\s}) - a\Tr(C^{\s} \ovx ) 
					+ \Tr(C^{\s}\ovx ) \Tr(C\ovx ^{\s}), \\
		bb^{\s} & \mapsto & bb^{\s} + b^{\s}\Tr(\ovC x) + b\Tr(\ovC ^{\s}x^{\s}) +
					\Tr(\ovC x)\Tr(\ovC ^{\s}x^{\s}), \\
		cc^{\s} & \mapsto & cc^{\s}, \\
		\Tr(A\ovA ^{\s}) & \mapsto & \Tr(A\ovA ^{\s}) + \Tr(AB^{\s}x^{\s}) + 
					\Tr(\ovx \ovB \ovA ^{\s}) + 
					\Tr((\ovx \ovB )(B^{\s}x^{\s})), \\
		\Tr(B\ovB ^{\s}) & \mapsto & \Tr(B\ovB ^{\s}) - 
					\Tr(\ovA \ovx ^{\s} \ovB ^{\s}) - \Tr(BxA^{\s}) + 
					\Tr((\ovA \ovx ^{\s})(xA^{\s})), \\
		\Tr(C\ovC ^{\s}) & \mapsto & \Tr(C\ovC ^{\s}) - 
					\Tr(C(\ovx ^{\s}C^{\s}\ovx )) - \Tr(C^{\s}(\ovx C\ovx ^{\s}))
					+ a\Tr(x\ovC ^{\s}) + \\
					& & + a^{\s}\Tr(C\ovx ^{\s}) -
					b^{\s}\Tr(C\ovx ) - b\Tr(x^{\s}\ovC ^{\s}).
	\end{array}
\end{equation}
Using Lemmas \ref{lemma:5_isotropic} and
\ref{lemma:5_oct_aux}, we get \mbox{$\Tr((\ovA \ovx ^{\s})(xA^{\s})) = 0= \Tr((\ovx \ovB )(B^{\s}x^{\s}))$}.
Next, we also obtain $\Tr(C( \ovx ^{\s} C^{\s} \ovx )) = \Tr(C \ovx ^{\s}
\Tr(C^{\sigma} \ovx )) = \Tr(C^{\sigma} \ovx )\Tr(C\ovx ^{\s})$. Likewise,
we get $\Tr(C^{\s}(\ovx C\ovx ^{\s})) = \Tr ((x^{\s}\ovC x)\ovC ^{\s}) = 
\Tr(\ovC ^{\s} (x^{\s}\ovC x)) = \Tr(\ovC ^{\s} x^{\s})\Tr(\ovC x)$. We see
that all the terms except $aa^{\s},bb^{\s},cc^{\s}$ and $\Tr(A\ovA ^{\s}),
\Tr(B\ovB ^{\s}), \Tr(C\ovC ^{\s})$ cancel out, so it follows that the elements
$N_x$ preserve the Hermitean form. Hence, we have shown the following. 

\begin{proposition}
	\label{prop:5_nx_2se}
	The matrices $N_x$, $N_x'$, and $N_x''$ where $x\ovx^{\s} = 0 = \ovx^{\s}x$ and $x, \ovx^{\s}$
	generate a sociable subalgebra, encode elements of ${}^2\SE_{6,K}(F)$. 
\end{proposition}

\begin{lemma}
	\label{lemma:5_xsyx_eq_xyxs}
	If $x \in \OO_K$ is such that $x \ovx^{\s} = 0 = \ovx^{\s} x$, with $x,\ovx^{\s}$ generating
	a sociable subalgebra, then $x^{\s} y x = x y x^{\s}$ for all $y \in \OO_K$.
\end{lemma}

\begin{proof}
	There is nothing to prove if $x=0$. Otherwise put $z=x^{\s}$.
	We shall show that $z$ is a scalar multiple of $x$.
	The hypotheses and their images under $\s$ give
	$x\ovz=0=\ovx z$.
	Write $f(u,v)=\Tr(u\ovv)$ for the non-degenerate polar form of the
	octonion norm. Polarising the identities
	$u(\ovu v)=\NN(u)v$ and $(v\ovu)u=\NN(u)v$ gives, for every $t\in\OO_K$,
	\[
	\begin{aligned}
		x(\ovt z)+t(\ovx z)&=f(x,t)z,\\
		(x\ovt)z+(x\ovz)t&=f(t,z)x.
	\end{aligned}
	\]
	Since $z=\Tr(z)1_{\OO}-\ovz$, sociability of $x,\ovz$, together
	with alternativity, gives $x(\ovt z)=(x\ovt)z$.
	The zero products above therefore reduce these two identities to
	\[
		f(x,t)z=f(t,z)x.
	\]
	As $x\neq0$ and $f$ is non-degenerate, we may choose $t$ with
	$f(x,t)=1$. Hence $z=\xi x$ for some $\xi\in K$. Consequently,
	for every $y\in\OO_K$,
	\[
		x^{\s}yx=\xi(xyx)=xyx^{\s},
	\]
	where $xyx$ is unambiguous by alternativity.
\end{proof}

Of great interest for us is the action of $N_x$ on $\J_{10}^{abC}$.

\begin{theorem}
	\label{theorem:5_nx_omega}
	The actions of the elements $N_x$ on $\J_{10}^{abC}$, where $x \in \OO_K$ is such that
	\mbox{$x \ovx^{\s} = 0 = \ovx^{\s} x$}, with $x,\ovx^{\s}$ generating a sociable subalgebra,
	generate a group of type $\Omega_{10,K}^{-}(F)$, as $x$ ranges through all suitable octonions
	in $\OO_K$.
\end{theorem}

The proof follows the same inductive strategy as the construction of $\Omega_{10}^+(F)$
in \cite{BSW1}: we build up $\Omega_4 \to \Omega_6 \to \Omega_8 \to \Omega_{10}$ by
successively adjoining elements. In the present case, the parameter $\lambda$ ranges
over $K$ rather than $F$, yielding the
minus-type group $\Omega_{10,K}^{-}(F)$.

\begin{lemma}
	\label{lemma:5_omega4minus}
	The actions of $N_{\lambda e_{\pm 1}}$ on the $4$-dimensional $K$-space
	$V_4$ spanned by the elements of the form $(a,b,0\mid 0,0,C)$ with $C \in \langle e_{-1}, e_1\rangle$,
	as $\lambda$ ranges over $K$, generate $\Omega_{4,K}^{-}(F)$.
\end{lemma}

\begin{proof}
	With respect to the basis $v_1 = (0,0,0\mid 0,0,e_{-1})$, $v_2 = (0,1,0\mid 0,0,0)$,
	$v_3 = (1,0,0\mid 0,0,0)$, $v_4 = (0,0,0\mid 0,0,e_1)$, the matrices of
	$N_{\lambda e_{-1}}$ and $N_{\lambda e_1}$ are
	\begin{equation*}
		\begin{bmatrix}
			1 & 0 & 0 & 0 \\
			-\lambda^{\s} & 1 & 0 & 0 \\
			\lambda & 0 & 1 & 0 \\
			-\lambda\lambda^{\s} & \lambda & -\lambda^{\s} & 1
		\end{bmatrix},\quad
		\begin{bmatrix}
			1 & \lambda & -\lambda^{\s} & -\lambda\lambda^{\s} \\
			0 & 1 & 0 & -\lambda^{\s} \\
			0 & 0 & 1 & \lambda \\
			0 & 0 & 0 & 1
		\end{bmatrix}.
	\end{equation*}
	These factor as Kronecker products:
	\begin{equation*}
		N_{\lambda e_{-1}} =
		\begin{bmatrix} 1 & 0 \\ \lambda & 1 \end{bmatrix} \otimes
		\begin{bmatrix} 1 & 0 \\ -\lambda^{\s} & 1 \end{bmatrix},\quad
		N_{\lambda e_1} =
		\begin{bmatrix} 1 & -\lambda^{\s} \\ 0 & 1 \end{bmatrix} \otimes
		\begin{bmatrix} 1 & \lambda \\ 0 & 1 \end{bmatrix}.
	\end{equation*}
	Put $D=\diag(1,-1)$. In each displayed tensor product, the second
	factor is $DS^{\s}D^{-1}$, where $S$ is the first factor. This relation
	is preserved under multiplication. Since the first factors run through
	all upper and lower elementary matrices of $\SL_2(K)$, the generated
	group is
	\begin{equation*}
		\{\,S\otimes DS^{\s}D^{-1}:S\in\SL_2(K)\,\}.
	\end{equation*}
	Thus we obtain a subgroup of the central product
	$\SL_2(K)\circ\SL_2(K)$ in which the second factor is determined by
	the first. This representation has kernel $\{\II_2,-\II_2\}$, since
	a tensor product equal to $\II_4$ forces $S$ to be scalar. Its image
	is therefore isomorphic to $\PSL_2(K)$.

	To identify the orthogonal action, consider the $4$-dimensional $F$-subspace
	\begin{equation*}
		W_4=\{\,(a,-a^{\s},0\mid0,0,r e_{-1}+s e_1):
		       a\in K,\ r,s\in F\,\},
	\end{equation*}
	which spans $V_4$ over $K$. Identify the displayed vector with the
	Hermitean matrix
	\begin{equation*}
		Y=\begin{bmatrix}r&a^{\s}\\a&-s\end{bmatrix},
		\qquad \det(Y)=-aa^{\s}-rs.
	\end{equation*}
	Thus the restriction to $W_4$ of $Q(a,b,0\mid0,0,C)=ab-C\ovC$
	is represented by the determinant. The right action of
	$S\otimes DS^{\s}D^{-1}$ becomes $Y\mapsto S^{\T}YS^{\s}$.
	Since the norm plane is anisotropic and the $(r,s)$-plane is hyperbolic,
	this quadratic form has Witt index $1$. The displayed action is the
	standard action of $\SL_2(K)$ onto $\Omega_{4,K}^{-}(F)$
	\cite[Corollary~12.42]{Taylor}. This also holds in characteristic $2$,
	where the kernel is trivial.
\end{proof}

\begin{proof}[Proof of Theorem~\ref{theorem:5_nx_omega}]
We extend $\Omega_{4,K}^{-}(F)$ to $\Omega_{10,K}^{-}(F)$ in three steps, following the
pattern established in \cite{BSW1}.

\textit{Step 1.}
Consider the $6$-space $V_6 = \{(a,b,0\mid 0,0,C) : C \in \langle e_{-1}, e_{\ombb},
e_{-\ombb}, e_1\rangle\}$. Our copy of $\Omega_{4,K}^{-}(F)$ fixes the singular isotropic
vectors $u_{\ombb} = (0,0,0\mid 0,0,e_{\ombb})$ and $u_{-\ombb} = (0,0,0\mid 0,0,e_{-\ombb})$.
The element $N_{e_{\ombb}}$ preserves $u_{\ombb}$ but moves $u_{-\ombb}$; adjoining it
gives a group of shape $F^4 \cn \Omega_{4,K}^{-}(F)$. The element $N_{e_{-\ombb}}$ does
not preserve $\langle u_{\ombb}\rangle$, so adjoining it yields $\Omega_{6,K}^{-}(F)$.

\textit{Step 2.}
Consider the $8$-space $V_8$ spanned by the elements \mbox{$(a,b,0\mid 0,0,C)$} with
$C \in \langle e_{-1}, e_{\ombb}, e_{\omega}, e_{-\omega}, e_{-\ombb}, e_1\rangle$. The vectors
$u_{\pm\omega} = (0,0,0\mid 0,0,e_{\pm\omega})$ are fixed by $\Omega_{6,K}^{-}(F)$. 
Adjoining $N_{e_{\omega}}$ (which preserves $u_{\omega}$
but not $u_{-\omega}$) gives $F^6 \cn \Omega_{6,K}^{-}(F)$; adjoining $N_{e_{-\omega}}$
(which moves $\langle u_{\omega}\rangle$) yields $\Omega_{8,K}^{-}(F)$.

\textit{Step 3.}
The full space $\J_{10}^{abC}$ contains the additional singular isotropic vectors
$u_{\pm 0} = (0,0,0\mid 0,0,e_{\pm 0})$, fixed by
$\Omega_{8,K}^{-}(F)$. Adjoining $N_{e_0}$ gives $F^8 \cn \Omega_{8,K}^{-}(F)$;
adjoining $N_{e_{-0}}$ yields $\Omega_{10,K}^{-}(F)$.
\end{proof}

\section{Action of ${}^2\SE_{6,K}(F)$ on white points}

As in the case of $\SE_6$ (see \cite{BSW1}), we are interested in the action on white points. We will, however, see
that although $\SE_6$ acts transitively on white points, the action of ${}^2\SE_{6,K}(F)$ splits
into several orbits. We say that a non-zero Albert vector $X$ is \textit{isotropic} if
$\H(X) = 0$. 

We first consider some examples. Suppose $X_1 = (0,0,0\mid 0,0,e_0)$. As we know (see \cite[Section 7]{BSW1}),
it determines 
a $17$-space $\{\ (a,b,0 \mid A,B,C)\ \big|\ e_0 A = B e_0 = \Tr(e_0 \ovC) = 0\ \}$. A
straightforward calculation shows that
this $17$-space $U_1$ is spanned by the Albert vectors of the form \mbox{$(a,b,0\mid A,B,C)$} with
\begin{equation}
	\begin{array}{r@{\;}c@{\;}l}
		A & \in & \langle e_{\bar{\omega}}, e_{\omega}, e_{-0}, e_1 \rangle, \\
		B & \in & \langle e_{-1}, e_{-0}, e_{-\omega}, e_{-\bar{\omega}} \rangle, \\
		C & \in & \langle e_{-1}, e_{\bar{\omega}}, e_{\omega}, e_0, e_{-\omega}, e_{-\bar{\omega}}, e_1
		\rangle.
	\end{array}
\end{equation}
We are also interested in the radical $R_1$ of $\H$ inside this $17$-space. In our case it is spanned
by the vectors of the form \mbox{$(0,0,0\mid A,B,C)$} with
\begin{equation}
	\begin{array}{r@{\;}c@{\;}l}
		A & \in & \langle e_{\bar{\omega}}, e_{\omega}, e_{-0}, e_1 \rangle, \\
		B & \in & \langle e_{-1}, e_{-0}, e_{-\omega}, e_{-\bar{\omega}} \rangle, \\
		C & \in & \langle e_0 \rangle.
	\end{array}
\end{equation}
In other words, our vector $X_1$ determines the $17$-space $U_1$ and the $9$-dimensional radical $R_1$ 
of $\H$ in $U_1$. Note that $X_1$ is isotropic with respect to $\H$, and also $X_1 \in R_1$. 

It turns out that there is another type of isotropic white vectors. Consider the vector
$X_2 = (0,0,0 \mid 0,0,e_0 + \lambda e_1)$, where $\lambda \in K\setminus F$. It again determines a 
$17$-space $U_2$, spanned by the Albert vectors of the form $(a,b,0\mid A,B,C)$ with 
\begin{equation}
	\label{eq:U2}
	\begin{array}{r@{\;}c@{\;}l}
		A & \in & \langle e_{\bar{\omega}}+\lambda e_{-\omega}, e_{\omega}-\lambda e_{-\bar{\omega}}
		, e_{-0}, e_1
		\rangle, \\
		B & \in & \langle e_{-1}+\lambda e_0, e_{-0} - \lambda e_1,
	e_{-\omega}, e_{-\bar{\omega}}  \rangle, \\
		C & \in & \langle e_{-1}-\lambda^2 e_1 - \lambda \cdot 1_{\OO}, e_{\bar{\omega}}, e_{\omega},
	e_0 + \lambda e_1, e_{-\omega}, e_{-\bar{\omega}}, e_1 \rangle.
	\end{array}
\end{equation}
We find that the radical $R_2$ of $\H$ in $U_2$ is spanned by the Albert vectors of the form 
$(0,0,0\mid A,B,C)$ with 
\begin{equation}
	\label{eq:R2}
	\begin{array}{r@{\;}c@{\;}l}
		A & \in & \langle e_{-0}, e_1 \rangle, \\
		B & \in & \langle e_{-\omega}, e_{-\bar{\omega}} \rangle, \\
		C & \in & \langle e_0 + \lambda^{\s} e_1 \rangle,
	\end{array}
\end{equation}
i.e. it is $5$-dimensional. We notice that $X_2$ is isotropic, but in this case $X_2 \not\in R_2$.

We conclude that the white points $\langle X_1 \rangle$ and $\langle X_2 \rangle$ belong to different
orbits under the action of ${}^2\SE_{6,K}(F)$. Of course, there is also at least one orbit on the
non-isotropic white points. Under hypothesis~\eqref{eq:1_norm_surjective},
we will now prove that these are the only three orbits.

\begin{lemma}
\label{lemma:6_reduce_J10}
Every white point is $ {}^2\SE_{6,K}(F)$-conjugate to a white point in
$\J_{10}^{abC}$.
\end{lemma}

\begin{proof}
	Let $X=(a,b,c\mid A,B,C)$ be white. If $A=B=C=0$, then exactly one
	scalar co\"ordinate is non-zero, and a power of $\tau$ gives the result.
	Otherwise, after applying a power of $\tau$, we may assume $C\neq0$.
	Choose a basis octonion $e_i$ with $\Tr(C\overline{e_i})\neq0$ and
	$t\in K$ such that $t^{\s}\Tr(C\overline{e_i})=a$.
	By~\eqref{eq:5_nx_image}, $N_{t e_i}$ makes $a=0$.
	All scalar multiples of basis octonions are admissible parameters:
	they have the required zero products, and lie with their conjugates
	in the sociable $K$-algebra spanned by $1_{\OO}$ and $e_i$.

	If now $B=C=0$, a power of $\tau$ again gives the result. If $C=0$
	but $B\neq0$, interchange the second and third matrix co\"ordinates,
	using an element generated by $\delta$ and $\tau$. This keeps $a=0$
	and replaces $C$ by $\ovB\neq0$. We may therefore assume
	\[
		a=0,\qquad C\neq0,\qquad BC=0.
	\]
	Write $A=\sum A_i e_i$, $B=\sum B_i e_i$, and $C=\sum C_i e_i$.
	We next arrange that $C_0\neq0$, using the elements $P_u$.
	If $C_0=0$ and $C_i\neq0$ for some $i\neq\pm0$, take
	$u=1_{\OO}+e_{-i}$: in $\ovu C\ovu$ the coefficient of $e_0$ is $C_i$.
	If only $C_{-0}$ is non-zero, take $u=e_1+e_{-1}$, which gives
	coefficient $-C_{-0}$. Both choices have norm $1$ and lie in $\OO_F$.
	Rescale the representative to obtain $C_0=1$.

	For $a=0$, formula~\eqref{eq:5_nx_image} gives
	\[
		B\mapsto B+\ovx\ovC,\qquad C\mapsto C-\ovB\ovx^{\s}
	\]
	under $N'_x$. The multiplication table gives the following changes:
	\[
	\begin{array}{c|c}
		x&\text{changed coefficient}\\ \hline
		t e_0&B_{-0}\mapsto B_{-0}+t\\
		t e_{-i},\quad i\in\{1,\omega,\ombb\}
			&B_{-i}\mapsto B_{-i}-t .
	\end{array}
	\]
	In each row, $C_0$ and the other three coefficients among
	$B_{-0},B_{-1},B_{-\omega},B_{-\ombb}$ are unchanged.
	Thus, successively taking the parameters
	\[
		-B_{-0}e_0,\quad B_{-1}e_{-1},\quad
		B_{-\omega}e_{-\omega},\quad B_{-\ombb}e_{-\ombb}
	\]
	makes these four coefficients zero. We still have $BC=0$.
	Now the coefficients of $e_0,e_1,e_{\omega},e_{\ombb}$ in $BC$ are
	respectively $B_0,B_1,B_{\omega},B_{\ombb}$, since $C_0=1$.
	Hence $B=0$. The whiteness relation $AB=c\ovC$ then gives $c=0$.

	Finally apply the elements $N''_x$. Since $a=c=0$ and $B=0$, these
	conditions remain true, while
	\[
		A\mapsto A-\ovC\ovx^{\s},\qquad C\mapsto C+\ovx\ovA.
	\]
	Here $x=t e_0$ changes $A_{-0}$ to $A_{-0}-t^{\s}$, and
	$x=t e_i$ changes $A_i$ to $A_i+t^{\s}$ for
	$i\in\{1,\omega,\ombb\}$. Again $C_0$ and the other three listed
	coefficients are unchanged. Taking successively
	\[
		A_{-0}^{\s}e_0,\quad -A_1^{\s}e_1,\quad
		-A_{\omega}^{\s}e_{\omega},\quad -A_{\ombb}^{\s}e_{\ombb}
	\]
	clears these four coefficients. The relation $CA=0$ then forces
	$A_0=A_{-1}=A_{-\omega}=A_{-\ombb}=0$, by inspecting these same
	coefficients of $CA$. We have reached
	$(0,b,0\mid0,0,C)\in\J_{10}^{abC}$.
\end{proof}

\begin{theorem}
\label{theorem:6_three_orbits}
Assume that the field norm $\NN_{K/F}:K^{\times}\rightarrow F^{\times}$ is
surjective, as in~\eqref{eq:1_norm_surjective}. Then the group
${}^2\SE_{6,K}(F)$ has precisely three orbits on white points. Representatives may be chosen as
\[
\begin{aligned}
	\langle X_1 \rangle &= \langle (0,0,0\mid 0,0,e_0)\rangle,\\
	\langle X_2 \rangle &= \langle (0,0,0\mid 0,0,e_0+\lambda e_1)\rangle,\  \lambda \in K\setminus F,\\
	\langle X_3 \rangle &= \langle (1,0,0\mid 0,0,0)\rangle.
 \end{aligned}
\]
More precisely, the three orbits are:
\begin{enumerate}[(i)]
	\item isotropic white points $\langle X \rangle$ such that $X$ lies in the radical of $\H$
		restricted to the $17$-space $U_X$ determined by $\langle X \rangle$
		(equivalently, $\dim(\rad(\H|_{U_X})) = 9$);
	\item isotropic white points $\langle X \rangle$ such that $X$ does not lie in this radical
		(equivalently, $\dim(\rad(\H|_{U_X})) = 5$);
	\item non-isotropic white points, i.e., $\H(X)\neq 0$.
\end{enumerate}

\end{theorem}

\begin{proof}
	We have already seen that these three types of white points are distinguished by 
	invariants preserved by ${}^2\SE_{6,K}(F)$, so they lie in different orbits. It remains 
	to show that each type forms a single orbit.

	By Lemma~\ref{lemma:6_reduce_J10}, it is enough to consider white points
	in $V=\J_{10}^{abC}$. On $V$ define
	\[
		Q((a,b,0\mid 0,0,C))=ab-C\ovC,
	\]
	and let $f$ be the polar form of $Q$, so that
	\[
		f((a,b,0\mid 0,0,C),(d,e,0\mid 0,0,F))=ae+bd-\Tr(C\ovF).
	\]
	Let $B=\B_{\H}|_{V}$, i.e.
	\[
		B((a,b,0\mid 0,0,C),(d,e,0\mid 0,0,F))=ad^{\s}+be^{\s}+\Tr(C\ovF^{\s}).
	\]

	Choose $\alpha\in K^{\times}$ such that $\alpha^{\s}=-\alpha$ (in characteristic $2$ one may take $\alpha=1$). Set
	\[
		U=\{(a,a^{\s},0\mid 0,0,\alpha D)\mid a\in K,\ D\in\OO_F\}.
	\]
	Then $U$ is an $F$-subspace of $V$, and $V=U\otimes_F K$.
	Moreover, for $X,Y\in U$ of the form
	\begin{equation*}
		\begin{array}{r@{\;}c@{\;}l}
			X&=&(a,a^{\s},0\mid 0,0,\alpha D),\\
			Y&=&(d,d^{\s},0\mid 0,0,\alpha E)
		\end{array}
	\end{equation*}
	we have
	\[
		f(X,Y)=ad^{\s}+a^{\s}d-\alpha^{2}\Tr(D\ovE)
		 =ad^{\s}+a^{\s}d+\alpha\alpha^{\s}\Tr(D\ovE)
		=B(X,Y).
	\]
	Thus $f_U=B_U$. Also
	\[
		Q_U(X)=aa^{\s}-\alpha^{2}D\ovD,
	\]
	so $Q_U$ is the orthogonal sum of the anisotropic norm plane of $K/F$
	and a split $8$-space. It is non-degenerate and has Witt index $4$.
	Hence $U$ is a $(Q,B)$-subspace of $V$.

	We next exhibit the required orthogonal subgroup of $ {}^2\SE_{6,K}(F)$.
	Put
	\[
		U_0=\alpha^{-1}U
		=\{(a,-a^{\s},0\mid0,0,C):a\in K,\ C\in\OO_F\}.
	\]
	Let $A_0$ be the $2$-dimensional $F$-subspace given by $C=0$,
	and let $W$ be the $8$-dimensional $F$-subspace given by $a=0$.
	Every vector of $U_0$ is uniquely a sum of a vector in $A_0$ and a
	vector in $W$. These subspaces are orthogonal with respect to $f$:
	$f(v,w)=0$ for all $v\in A_0$ and $w\in W$. On $U_0$ the form $Q$ becomes
	$-aa^{\s}-\NN(C)$, so $A_0$ is anisotropic and $W$ is split.
	The elements $P_u$ with $u\in\OO_F$ and $\NN(u)=1$ induce
	$\Omega(W)$ and fix $A_0$ pointwise, by \cite[Section~6]{BSW1}.
	For $u,v\in U_0$ with $Q(u)=f(u,v)=0$, write
	\[
		z^{\rho_{u,v}}=
		z+f(z,v)u-f(z,u)v-Q(v)f(z,u)u.
	\]
	These are the Siegel transformations of the orthogonal space
	\cite[Section~11]{Taylor}. For $\mu\in K$ and $i\in\pm I$,
	substitution in~\eqref{eq:5_nx_image} identifies the action on $U_0$ as
	\[
		N_{\mu e_i}=\rho_{u_i,v_\mu},\qquad
		u_i=(0,0,0\mid0,0,e_i),\quad
		v_\mu=(-\mu^{\s},\mu,0\mid0,0,0).
	\]
	Let $D$ be the group induced on $U_0$ by these elements and the $P_u$.
	Conjugating by $\Omega(W)$ and using
	$\rho_{u,v}\rho_{u,w}=\rho_{u,v+w}$ shows that $D$ contains
	$\rho_{u,v}$ for every non-zero singular $u\in W$ and every $v\in u^\perp$:
	the parameters in $A_0$ come from the displayed $N_{\mu e_i}$,
	and those in $W\cap u^\perp$ already belong to $\Omega(W)$.
	If $z=w+v_0$ is a non-zero singular vector of $U_0$, with
	$w\in W$, $v_0\in A_0$ and $v_0\neq0$, then $w\neq0$.
	Choose a singular $u\in W$ with $r=f(w,u)\neq0$;
	the transformation $\rho_{u,v_0/r}$ removes the $A_0$-component of $z$.
	Thus $D$ is transitive on singular points, since $\Omega(W)$ is so.
	Conjugation now supplies all the Siegel transformations of $U_0$.
	By \cite[Theorem~11.46]{Taylor}, they generate $\Omega(U_0)$.
	Multiplication by $\alpha$ therefore identifies the induced group
	with $\Omega(U,Q)$.

	The orbit arguments in Section~3 use isometries of $U$ prescribed on
	a totally singular subspace of dimension at most $2$, or on a
	non-degenerate plane. Their target subspace is orthogonal to a
	hyperbolic plane $L$ in $U$, since the Witt index is $4$.
	The proof of \cite[Theorem~11.46]{Taylor} gives
	$\GO(U,Q)=\GO(L,Q)\Omega(U,Q)$, where $\GO(L,Q)$ acts trivially on
	$L^\perp$. Multiplying a chosen Witt isometry on the right by a
	suitable element of $\GO(L,Q)$ therefore puts it in $\Omega(U,Q)$
	without changing its prescribed images. Consequently the orbit
	statements of Proposition~\ref{prop:3_2forms_orbits} and
	Lemma~\ref{lemma:3_2forms_nonisotropic} hold for this subgroup as well.

	Now let $X=(a,b,0\mid 0,0,C)\in V$ be white. Then $Q(X)=0$. If $\H(X)=0$, then $X$ is singular isotropic for the pair $(Q,B)$, so by
	Proposition~\ref{prop:3_2forms_orbits} there are exactly two orbits, with representatives
	\[
		\begin{array}{l}
			\langle (0,0,0\mid 0,0,e_0)\rangle,\\
			\langle (0,0,0\mid 0,0,e_0+\lambda e_1)\rangle,
		\end{array}
	\]
	where $\lambda\in K\setminus F$.

	If $\H(X)\neq0$, hypothesis~\eqref{eq:1_norm_surjective} gives
	$t\in K^{\times}$ with $tt^{\s}=\H(X)^{-1}$. Replacing $X$ by $tX$
	does not change its point and gives $\H(tX)=1$.
	This is the norm-surjectivity step in the non-isotropic case.
	Lemma~\ref{lemma:3_2forms_nonisotropic} now shows that all such points form
	a single orbit.
	Since
	\[
		\begin{array}{r@{\;}c@{\;}l}
			Q((1,0,0\mid 0,0,0))&=&0,\\
			\H((1,0,0\mid 0,0,0))&=&1\neq0,
		\end{array}
	\]
	this orbit may be represented by
	\[
		\langle (1,0,0\mid 0,0,0)\rangle.
	\]

	Since ${}^2\SE_{6,K}(F)$ preserves both $\H$ and the associated $17$-space $U_X$,
	the quantities $\dim(\rad(\H|_{U_X}))$ and the property $X\in\rad(\H|_{U_X})$
	are orbit invariants. The characterisation in terms of radicals follows from the computations at the start of
	this section: $X_1$ lies in the $9$-dimensional radical $R_1$ of $\H$ on its $17$-space $U_1$,
	while $X_2$ does not lie in the $5$-dimensional radical $R_2$ of $\H$ on its $17$-space $U_2$.
	To verify $X_2 \notin R_2$, note that $R_2$ has $C$-co\"{o}rdinate in
	$\langle e_0 + \lambda^{\s} e_1\rangle$, whereas $X_2$ has $C = e_0 + \lambda e_1$.
	Since $\lambda \notin F$, these $1$-spaces are distinct, so $X_2 \notin R_2$.
\end{proof}

\section{Stabilisers of white vectors}

We are now interested in the stabiliser in ${}^2\SE_{6,K}(F)$ of 
$X_3 = (0,0,1\mid 0,0,0)$, which is non-isotropic. As we know, the elements
$N_x$ with $x\ovx^{\s} = \ovx^{\s}x = 0$ preserve $X_3$ (\ref{eq:5_nx_image}).
We now prove the following theorem.
\begin{theorem}
	\label{theorem:7_type3_stabiliser}
	The stabiliser in ${}^2\SE_{6,K}(F)$ of $X_3 = (0,0,1\mid 0,0,0)$ is a subgroup of shape
	\begin{equation}
		\Spin_{10,K}^{-}(F).
	\end{equation}
\end{theorem}

\begin{proof}
	Let $G$ be the stabiliser of $p=X_3$, and put $W=\J_{10}^{abC}$.
	For $X=(a,b,c\mid A,B,C)$, since $G$ fixes $p$ and preserves $\fdet$,
	it preserves
	\[
		\fdet(X+p)-\fdet(X)=ab-\NN(C).
	\]
	The radical of this quadratic form is $U_3=\J_{17}^{cAB}$.
	Since $G$ also preserves $\B_{\H}$, it stabilises
	$W=U_3^\perp$ and $\J_{16}^{AB}=U_3\cap p^\perp$.
	Thus $G$ preserves the decomposition
	\[
		\J=Kp\oplus\J_{16}^{AB}\oplus W.
	\]
	On $W$ write $Q=ab-\NN(C)$, and let $f$ be its polar form.

	Consider the $F$-subspace
	\[
		V=\{(a,-a^{\s},0\mid0,0,C):a\in K,\ C\in\OO_F\}.
	\]
	Then $W=V\otimes_F K$, and direct calculation gives
	\[
		Q((a,-a^{\s},0\mid0,0,C))=-aa^{\s}-\NN(C),
		\qquad f|_V=-\B_{\H}|_V.
	\]
	In particular, $V$ is a $(Q,-\B_{\H})$-subspace: its quadratic form
	is the orthogonal sum of the anisotropic norm plane $A_0$ given by
	$C=0$ and the split $8$-space $W_0$ given by $a=0$.
	Proposition~\ref{prop:3_2forms}, applied to $-\B_{\H}$, therefore
	shows that $G$ preserves $V$ and induces a subgroup of $\GO(V,Q)$.

	Let $L$ be the subgroup generated by the elements $P_u$, with
	$u\in\OO_F$ and $\NN(u)=1$, and the elements $N_{\mu e_i}$, with
	$\mu\in K$ and $i\in\pm I$. These elements fix $p$, so $L\leqslant G$.
	The proof of Theorem~\ref{theorem:6_three_orbits} shows that their
	induced group on $V$ is $\Omega(V,Q)=\Omega_{10,K}^{-}(F)$.
	Moreover, Lemma~9.5(c) of \cite{BSW1} shows that an element of $G$
	acting trivially on $W$ is $P_1$ or $P_{-1}$. Both belong to $L$;
	in characteristic $2$ they coincide. Thus it remains to show that
	every element of $G$ induces an element of $\Omega(V,Q)$.

	Here we use hypothesis~\eqref{eq:1_norm_surjective}.
	Since $Q|_{A_0}=-\NN_{K/F}$, the plane $A_0$ represents every
	non-zero element of $F$. Consequently
	\[
		\GO(V,Q)=\GO(A_0,Q)\Omega(V,Q),
	\]
	where $\GO(A_0,Q)$ acts trivially on $W_0$.
	Indeed, $\GO(V,Q)$ is generated by reflexions
	\cite[Corollary~11.42]{Taylor}. For every non-singular $v\in V$,
	choose $w\in A_0$ with $Q(w)=Q(v)$. Witt's theorem conjugates the
	reflexion in $v$ to the reflexion in $w$. Conjugate elements have
	the same image modulo the commutator subgroup
	$\Omega(V,Q)$ \cite[Theorem~11.45]{Taylor}, so each generating reflexion lies in
	$\GO(A_0,Q)\Omega(V,Q)$.

	Now take any $g\in G$. Multiplying $g$ on the right by a suitable
	element of $L$, we obtain $s\in G$ whose action on $V$ lies in
	$\GO(A_0,Q)$. Its $K$-linear action on $W$ fixes $\J_8^C$
	pointwise and preserves the form $ab$ on $\J_1^a\oplus\J_1^b$.
	It therefore either fixes or interchanges the two singular
	$1$-spaces $\J_1^a$ and $\J_1^b$.
	Lemma~9.5(b) of \cite{BSW1} excludes the interchange.
	The calculation in the proof of Lemma~9.5(a) gives the whole
	action of $s$, not just its action on $W$:
	\[
		(a,b,c\mid A,B,C)^s
		=(\beta^2a,\beta^{-2}b,c\mid\beta^{-1}A,\beta B,C)
	\]
	for some $\beta\in K^\times$.
	Preservation of the non-degenerate sesquilinear form on $\J_8^B$
	now gives $\beta\beta^{\s}=1$.
	Hence the map
	\[
		(a,-a^{\s},0\mid0,0,C)
		\longmapsto(\beta a,-(\beta a)^{\s},0\mid0,0,C)
	\]
	is an isometry of $V$, and its square is the action of $s$.
	Every square in $\GO(V,Q)$ belongs to $\Omega(V,Q)$:
	the quotient by the commutator subgroup is abelian and is
	generated by the images of involutory reflexions.
	Thus $s$, and hence $g$, induces an element of $\Omega(V,Q)$.
	Since $L$ supplies every such action and contains the kernel,
	we conclude that $G=L$.

	Finally, we identify this group with the spin group, using the
	split spin action on $\J$ constructed in \cite[Theorem~9.6]{BSW1}.
	The generators of $L$ belong to that copy of $\Spin_{10}^{+}(K)$:
	for $x=\mu e_i$ we have $N_x=M_xL_{-\ovx^{\s}}$, and the $P_u$
	belong to the same group. The equality $W=V\otimes_F K$ embeds
	the spin group of $(V,Q)$ in this split spin group, compatibly
	with its action on $W$. Its image on $V$ is $\Omega(V,Q)$,
	and its kernel is $\{P_1,P_{-1}\}$. Thus it consists of all the
	lifts of $\Omega(V,Q)$ in this same split spin group.
	We have proved that $G=L$ has exactly these lifts. Therefore
	\[
		G\cong\Spin_{10,K}^{-}(F),
	\]
	with $\Spin_{10,K}^{-}(F)\cong\Omega_{10,K}^{-}(F)$ in
	characteristic $2$.
\end{proof}

We are ready to investigate the stabiliser in ${}^2\SE_{6,K}(F)$ of 
\mbox{$X_1 = (0,0,0\mid 0,0,e_0)$}. From the previous section we know that the stabiliser of 
$X_3$ is a subgroup of shape $\Spin_{10,K}^{-}(F)$. By stabilising $X_3$ and $X_1$ simultaneously,
with the help of Lemma~2.1 in \cite{BSW1}, we find that their joint stabiliser is a group
of shape $F^8\cn\Spin_{8,K}^{-}(F)$.

To find the rest of the stabiliser, we consider the elements $N_{\mu x}$,
$N_{\mu x}'$, and $N_{\mu x}''$.  The ones that preserve $X_1$ but
need not preserve $X_3$ are
\begin{equation*}
	N_{\mu x}'\quad\hbox{with}\quad
	x\in\{e_{\bar{\omega}},e_\omega,e_{-0},e_1\},
\end{equation*}
and
\begin{equation*}
	N_{\nu y}''\quad\hbox{with}\quad
	y\in\{e_{-1},e_{-0},e_{-\omega},e_{-\bar{\omega}}\},
\end{equation*}
where $\mu,\nu\in K$.  Denote by $H$ the subgroup of
${}^2\SE_{6,K}(F)$ generated by the actions of all these elements on
$\J$.  Their action on $X_3$ is given by
\begin{equation}
	\begin{array}{r@{\;}c@{\;}l}
		N_{\mu x}': (0,0,1\mid 0,0,0) & \mapsto & (0,0,1\mid -\mu^{\s} x, 0, 0), \\
		N_{\nu y}'': (0,0,1\mid 0,0,0) & \mapsto & (0,0,1\mid 0, \nu y, 0).  
	\end{array}
\end{equation}
Consider the subspace $U_1^{\perp}$, where $\perp$ is taken with respect to
the sesquilinear form $\B_{\H}$. We find that $U_1^{\perp}$ is $10$-dimensional and
spanned by the vectors $(0,0,c\mid A,B,C)$ with
\begin{equation}
	\begin{array}{r@{\;}c@{\;}l}
		A & \in & \langle e_{\bar{\omega}}, e_{\omega}, e_{-0}, e_1 \rangle, \\
		B & \in & \langle e_{-1}, e_{-0}, e_{-\omega}, e_{-\bar{\omega}} \rangle, \\
		C & \in & \langle e_0 \rangle.
	\end{array}
\end{equation}
We first show that $U_1^{\perp}$ is an image in $\SE_6(K)$
of $\J_{10}^{abC}$, so that we can use the geometry 
described in \cite[Section 1]{BSW1} to pin down
the stabiliser of type $1$ vector to a certain subgroup.

Consider an element $M$ of $\SE_6(K)$, defined as
\begin{equation*}
	M = \begin{bmatrix}[c]
		- e_{-0} & 1 & 0 \\
		0 & 0 & 1 \\
		e_0 & e_{-0} & 0
	\end{bmatrix}.
\end{equation*}
It is straightforward to see that $M$ indeed preserves
the Dickson--Freudenthal determinant. The action of
$M$ on an arbitrary vector $(a,b,0\mid 0,0,C) \in \J_{10}^{abC}$
is given by
\begin{equation*}
	\begin{array}{r@{\;}c@{\;}l}
		(a,b,0\mid 0,0,C) & \mapsto & (0,a,b \mid A, B, -a e_0),\\
	\end{array}
\end{equation*}
where 
\begin{equation*}
	\begin{array}{r@{\;}c@{\;}l}
			A & \in & \langle e_{-1}, e_{\bar{\omega}},
						e_{\omega}, e_0 \rangle, \\
		B & \in & \langle e_{-1}, e_{-0}, e_{-\omega}, 
						e_{-\bar{\omega}} \rangle.	
	\end{array}
\end{equation*}
Acting further by $M_{e_{-0}}$, we obtain a vector
$(0,0,b \mid \hat{A}, \hat{B}, -a e_0)$ with
\begin{equation*}
	\begin{array}{r@{\;}c@{\;}l}
		\hat{A} & \in & \langle e_{\bar{\omega}}, e_{\omega}, e_{-0}, e_1 \rangle, \\
		\hat{B} & \in & \langle e_{-1}, e_{-0}, e_{-\omega}, e_{-\bar{\omega}} \rangle.
	\end{array}
\end{equation*}
In other words, we have shown that $U_1^{\perp}$ is indeed
an image in $\SE_6(K)$ of $\J_{10}^{abC}$. 
Now, $X_1 \in U_1^{\perp}$, so the
stabiliser of $\langle X_1 \rangle$ is no bigger than  
\mbox{$K^{16}\cn K^8 \cn \Spin_8^+(K).K^{\times}.K^{\times}$}.

Note that $X_3 \in U_1^{\perp}$, so the stabiliser in 
${}^2\SE_{6,K}(F)$ sends $X_3$ to some vector
in $U_1^{\perp}$. 
The actions on $U_1^{\perp}$ of the elements $N_{\mu x}'$ and $N_{\nu y}''$ from above are given
by
\begin{equation}
	\begin{array}{r@{\;}c@{\;}l}
		N_{\mu x}' : (0,0,c\mid A,B,C) & \mapsto & 
			(0,-\mu^{\s}\Tr(A\ovx^{\s}), c+\mu\Tr(\ovA x) \mid \\
		& & \mid A - \mu\mu^{\s} x^{\s} \ovA x - \mu^{\s} c x^{\s}, \\
		& &		 B + \mu\ovx\ovC, C - \mu^{\s} \ovB \ovx^{\s}), \\
				 
		N_{\nu y}'' : (0,0,c\mid A,B,C) & \mapsto & 
			(\nu\Tr(\ovB y), 0, c-\nu^{\s}\Tr(B\ovy^{\s}) \mid \\
		& & \mid A - \nu^{\s}\ovC\ovy^{\s}, \\
		& & B - \nu\nu^{\s} y^{\s}\ovB y + \nu c y, C + \nu\ovy\ovA). 
	\end{array}
\end{equation}
Since $x \in \{e_{\bar{\omega}}, e_{\omega}, e_{-0}, e_1\}$, we find $\Tr(A\ovx^{\s}) = 
\Tr(\ovA x) = 0 = x^{\s} \ovA x = \ovx \ovC$, and also $x^{\s} = x$, $\ovB \ovx^{\s} \in 
\langle e_0 \rangle$. Similarly, $\Tr(\ovB y) = \Tr(B\ovy^{\s}) = 0 = y^{\s} \ovB y = \ovC \ovy^{\s}$,
and $y^{\s} = y$, $\ovy \ovA \in \langle e_0 \rangle$ as $y \in \{ e_{-1}, e_{-0}, e_{-\omega}, 
e_{-\bar{\omega}} \}$. After these simplifications the action takes the form
\begin{equation}
	\begin{array}{r@{\;}c@{\;}l}
		N_{\mu x}' : (0,0,c\mid A,B,C) & \mapsto &
			(0,0,c\mid A  - \mu^{\s} c x, B, C - \mu^{\s}\ovB\ovx ), \\
		
		N_{\nu y}'' : (0,0,c\mid A,B,C) & \mapsto &
			(0,0,c\mid A, B + \nu c y, C + \nu \ovy \ovA).
	\end{array}
\end{equation}
The value of $\H$ on $U_1^{\perp}$ is given by
\begin{equation}
	\H((0,0,c\mid A,B,C)) = c c^{\s},
\end{equation}
so we conclude that the elements of the stabiliser of $X_1$, 
which do not preserve $X_3$, 
send $X_3$ to a vector of the form $(0,0,c\mid A,B, \ovB \ovA 
c^{-1})$, where $c$ is such that $c c^{\s} = 1$. In particular,
$c \neq 0$, so we can repeatedly use the elements $N_{\mu x}'$
and $N_{\nu y}''$ which preserve $X_1$ to map any vector
$(0,0,c\mid A,B, \ovB \ovA c^{-1})$ to $(0,0,c \mid 0,0,0)$,
whose span is $\langle X_3\rangle$.  Thus $H$ is transitive on the
corresponding $1$-spaces.

Denote by $G_1$ the stabiliser of $X_1$ and by $G_{13}$ the
joint stabiliser of $X_1$ and $X_3$. Let $\mathcal{O}$ and
$\hat{\mathcal{O}}$ be
the following sets:
\begin{equation}
	\begin{array}{r@{\;}c@{\;}l}
	\mathcal{O} & = & 
	\left\{\ W \in U_1^{\perp}\ \big|\ \mbox{$W$ is white and
		$\H(W)=1$ } \right\}, \\
		
	\hat{\mathcal{O}} & = & \left\{ 
		\ \langle W \rangle\ \big|\ W \in \mathcal{O}\ 
	\right\}.
	\end{array}
\end{equation}
Further, denote by $\hat{G}_{13}$ the joint stabiliser
of $X_1$ and $\langle X_3 \rangle$. Now, $\hat{\mathcal{O}}$
is $G_1$-invariant, and the subgroup $H \leqslant G_1$
acts transitively on it, which means 
$\langle X_3 \rangle^{G_1} = \langle X_3 \rangle^H$.

We first note that $\hat G_{13}=G_{13}$.  Put
\begin{equation*}
	 k_1=M M_{e_{-0}},
\end{equation*}
so that $\bigl(\J_{10}^{abC}\bigr)^{k_1}=U_1^{\perp}$.  Direct calculation gives
\begin{equation*}
	 p:=X_1^{k_1^{-1}}=(-1,0,0\mid0,0,0),\qquad
	 q:=X_3^{k_1^{-1}}=(0,1,0\mid0,0,0).
\end{equation*}
We shall use one more co\"{o}rdinate of this calculation.  If
$X=(a,b,c\mid A,B,C)$ and $B=\sum_{i\in\pm I}\beta_i e_i$, write
\begin{equation*}
	 X^M=(a_M,b_M,c_M\mid A_M,B_M,C_M),\qquad
	 X^{k_1}=(a_1,b_1,c_1\mid A_1,B_1,C_1).
\end{equation*}
Using $X^M=\ovM^{\T} X M$, we obtain
\begin{equation*}
	 a_M=-\beta_{-0},
\end{equation*}
and the $e_{-0}$-co\"{o}rdinate of $C_M$ is
$c+\beta_{-0}$.
The element $M_{e_{-0}}$ replaces $C$ by $C+a e_{-0}$, and hence
the $e_{-0}$-co\"{o}rdinate of $C_1$ is
$c+\beta_{-0}-\beta_{-0}=c$.

Transport the sesquilinear form by $k_1$, putting
\begin{equation*}
	 \B_1(Y,Z)=\B_{\H}(Y^{k_1},Z^{k_1}).
\end{equation*}
By (\ref{eq:4_B_sesquilinear}) and the trace pairing, in which
$e_0$ pairs only with $e_{-0}$, we have
\begin{equation}
	 \label{eq:7_type1_k1_pairing}
	 \B_1(p,X)=\B_{\H}(X_1,X^{k_1})
	 =\Tr\!\left(e_0\ovC_1^{\s}\right)=c^{\s}.
\end{equation}

Let $\hat g\in\hat G_{13}$ and put $g=k_1 \hat{g} k_1^{-1}$.
Then $p^{g}=p$ and $q^{g}=\lambda q$ for some
$\lambda\in K^{\times}$, while $g$ preserves $\B_1$ and
$\fdet$.  For an arbitrary $X=(a,b,c\mid A,B,C)$, write
\begin{equation*}
	 X^{g}=(a',b',c'\mid A',B',C').
\end{equation*}
It follows from (\ref{eq:7_type1_k1_pairing}) that
\begin{equation*}
	 (c')^{\s}=\B_1(p,X^{g})
	 =\B_1(p^g,X^{g})
	 =\B_1(p,X)=c^{\s}.
\end{equation*}
Thus
\begin{equation}
	 \label{eq:7_type1_c_fixed}
	 c'=c.
\end{equation}

Finally, direct substitution in the Dickson--Freudenthal determinant
gives
\begin{equation*}
	 \fdet(X+p+q)-\fdet(X+p)-\fdet(X+q)+\fdet(X)=-c.
\end{equation*}
Indeed, only the term $abc$ contains both the change in the
$a$-co\"{o}rdinate and the change in the $b$-co\"{o}rdinate.
Replacing $q$ by $\lambda q$ multiplies the right-hand side by
$\lambda$.  Therefore preservation of $\fdet$ gives
\begin{align*}
	 -c
	 &=\fdet(X^{g}+p+\lambda q)-\fdet(X^{g}+p)\\
	 &\hspace{18mm}-\fdet(X^{g}+\lambda q)+\fdet(X^{g})\\
	 &=-\lambda c'=-\lambda c,
\end{align*}
where the last equality follows from (\ref{eq:7_type1_c_fixed}).
Taking, for example, $X=(0,0,1\mid0,0,0)$ gives $\lambda=1$.
Thus $g$ fixes $q$, and hence $\hat{g}$ fixes $X_3$.  This
proves that $\hat G_{13}=G_{13}$. Notice that the argument is also 
valid in characteristic $2$.

Put $Z=H\cap G_{13}$, and let $Z_0$ be the subgroup of $H$ generated by
all commutators
\begin{equation*}
	 [N_{\lambda x}',N_{\mu y}'],\qquad
	 [N_{\lambda x}',N_{\mu y}''],\qquad
	 [N_{\lambda x}'',N_{\mu y}'']
\end{equation*}
with $x$ and $y$ in the appropriate displayed sets.  Direct substitution
in (\ref{eq:5_nx_image}) shows that every such commutator lies in $Z$ and
commutes with each of the eight displayed families.  Since these families
generate $H$, we have
\begin{equation*}
	 Z_0\leqslant Z\cap Z(H).
\end{equation*}
The formulae above show that each family is additive.  Since the eight
families commute modulo $Z_0$, collecting the factors shows that every
element of $H$ can be written as
\begin{equation*}
	 z\prod_{i=1}^{4}N'_{\mu_i x_i}
	       \prod_{j=1}^{4}N''_{\nu_j y_j},
	 \qquad z\in Z_0,
\end{equation*}
where the $x_i$ and $y_j$ run through the two displayed sets of four
basis elements.  Since $H\leqslant G_1$, an element of $H$ belongs to
$Z$ precisely when it fixes $X_3$.  Since $z\in Z$, it fixes $X_3$.
The image of $X_3$ under the ordered product has respectively
\begin{equation*}
	 -\sum_{i=1}^{4}\mu_i^{\s}x_i
	 \quad\hbox{and}\quad
	 \sum_{j=1}^{4}\nu_jy_j
\end{equation*}
as its $A$- and $B$-co\"{o}rdinates.  The $x_i$ and $y_j$ are linearly
independent, so an element in the displayed form belongs to $Z$ if and
only if all eight parameters vanish.  It follows that $Z=Z_0$ and that
\begin{equation*}
	 H/Z_0\cong(K,+)^8.
\end{equation*}
Now $Z_0\leqslant[H,H]$ because $Z_0$ is generated by commutators, while
$[H,H]\leqslant Z_0$ because $H/Z_0$ is abelian.  Hence $Z_0=[H,H]$.
For example, with respect to the basis
\begin{equation}
	\begin{array}{r@{\;}c@{\;}l}
		w_1 & = & (0,0,0 \mid 0,0,e_0), \\
		w_2 & = & (0,0,0 \mid 0,0,e_{\omega}), \\
		w_3 & = & (0,0,0 \mid 0,0,e_{\bar{\omega}}), \\
		w_4 & = & (0,0,0 \mid 0,0,e_{-1}), \\
		w_5 & = & (1,0,0 \mid 0,0,0), \\
		w_6 & = & (0,1,0 \mid 0,0,0), \\
		w_7 & = & (0,0,0 \mid 0,0,e_1), \\
		w_8 & = & (0,0,0 \mid 0,0,e_{-\bar{\omega}}),\\ 
		w_9 & = & (0,0,0 \mid 0,0,e_{-\omega}), \\
		w_{10} & = & (0,0,0 \mid 0,0,e_{-0}), \\
	\end{array}
\end{equation}
the commutator $[N_{\lambda e_{\bar{\omega}}}', 
N_{\mu e_{\omega}}']$ is represented by the matrix 
\begin{equation}
\II_{10} + (\lambda \mu^{\s} + \lambda^{\s}\mu)E_{7,1} 
- (\lambda \mu^{\s} + \lambda^{\s}\mu) E_{10,4},
\end{equation}
where 
$E_{i j}$ as usual denotes the matrix with $1$ in the row $i$
and column $j$ and zeroes in any other position. As we know,
$\lambda \mu^{\s} + \lambda^{\s}\mu
=\Tr_{K/F}(\lambda\mu^{\s})\in F$.  Running through all possible pairs
shows that $Z_0$ has precisely eight independent $F$-parameters.  Consequently
\begin{equation*}
	 Z=Z_0=[H,H]\cong(F,+)^8\leqslant Z(H).
\end{equation*}
We have therefore obtained the exact sequence
\begin{equation}
	 \label{eq:7_type1_kernel}
	 1\longrightarrow(F,+)^8\longrightarrow H\longrightarrow(K,+)^8
	 \longrightarrow1.
\end{equation}
Since $K/F$ is separable, $\Tr_{K/F}$ is onto, also in characteristic
$2$, and the displayed commutator is non-trivial for suitable
$\lambda,\mu$.  Thus $H$ is non-abelian.  If
(\ref{eq:7_type1_kernel}) split, then centrality of $Z$ would give the
abelian direct product $H\cong Z\times(K,+)^8$, a contradiction.  Hence
$H$ has non-split shape $F^8.K^8$.

The joint-stabiliser calculation identifies $Z$ with the normal $F^8$
in $G_{13}=Z\rtimes L$, where $L\cong\Spin_{8,K}^{-}(F)$.
The equality of the orbits of $\langle X_3\rangle$, together with
$\hat G_{13}=G_{13}$, gives $G_1=G_{13}H$.  Hence $G_1=HL$ and
$H\cap L=1$.  Since $Z$ is normal in $G_{13}$ and central in $H$,
it is normal in $G_1$.  Let $T$ be the subgroup of $G_1$ consisting
of the elements that commute with every element of $Z$; then
$T\lhd G_1$.  By the matrices in the proof of
\cite[Lemma~2.1]{BSW1}, $L$ acts on the eight parameters of $Z$
through its natural orthogonal action.  This action has trivial kernel
in characteristic $2$, and kernel of order $2$ otherwise.  Since
$G_1=HL$ and $H$ centralises $Z$, it follows that $T=H$ in
characteristic $2$, while $H$ has index $2$ in $T$ otherwise.  In the
latter case every square in $T$ belongs to $H$, and every displayed
generator of $H$ is a square in its additive family.  Thus $H$ is
precisely the subgroup generated by the squares in $T$.  Conjugation
by $G_1$ preserves $T$ and therefore preserves this subgroup.
Consequently $H\lhd G_1$ in every characteristic, and
$G_1=H\rtimes L$.
We have proved the following theorem.
\begin{theorem}
	\label{theorem:7_type1_stabiliser}
	The stabiliser in ${}^2\SE_{6,K}(F)$ of 
	$X_1 = (0,0,0\mid 0,0,e_0)$ 
	is a group of shape
	\begin{equation}
		(F^{8}.K^{8})\cn\Spin_{8,K}^{-}(F).
	\end{equation}
	Here the extension $F^8.K^8$ is non-split, whereas the displayed
	outer extension by $\Spin_{8,K}^{-}(F)$ is split.
\end{theorem}

Finally, we determine the stabiliser of $(0,0,0\mid 0,0,e_0 + \lambda e_1)$. 
\begin{theorem}
	\label{theorem:7_type2_stabiliser}
	The stabiliser in ${}^2\SE_{6,K}(F)$ of
	$X_2 = (0,0,0\mid 0,0,e_0 + \lambda e_1)$ is a group of shape
	\begin{equation}
		(F.K^{10})\cn\SU_{5,K}(F).
	\end{equation}
	Here the extension $F.K^{10}$ is non-split.
\end{theorem}

\begin{proof}
	Let $G_2$ denote the stabiliser of $X_2$ in ${}^2\SE_{6,K}(F)$. Since $G_2$ is a subgroup of $\SE_6(K)$,
	it preserves the Dickson--Freudenthal determinant and therefore stabilises the $17$-space $U_2$ determined
	by $X_2$ (\ref{eq:U2}). Since $G_2$ also preserves the Hermitean form $\H$, it stabilises the subspace
	$U_2^{\perp}$, where $\perp$ is again taken with respect to the sesquilinear form $\B_{\H}$.

	We now determine $U_2^{\perp}$ explicitly. As $\B_{\H}$ is non-degenerate on $\J$, we have
	$\dim(U_2^{\perp}) = 27 - 17 = 10$. Note that for $x = \sum_{i\in\pm I} x_i e_i$ and
	$y = \sum_{i\in\pm I} y_i e_i$ in $\OO_K$ we have
	\begin{equation}
		\Tr(x\ovy^{\s}) = \sum_{i\in\pm I} x_i y_{-i}^{\s},
	\end{equation}
	so $e_i$ pairs non-trivially with $e_{-i}$ only.  Applying this to the
	spanning vectors in (\ref{eq:U2}), we find that $U_2^{\perp}$ consists of
	the vectors $(0,0,c\mid A,B,C)$ with
	\begin{equation}
		\label{eq:U2perp}
		\begin{array}{r@{\;}c@{\;}l}
			A & \in & \langle e_{\bar{\omega}}+\lambda^{\s} e_{-\omega},
			e_{\omega}-\lambda^{\s} e_{-\bar{\omega}}, e_{-0}, e_1 \rangle, \\
			B & \in & \langle e_{-1}+\lambda^{\s} e_0, e_{-0} - \lambda^{\s} e_1,
			e_{-\omega}, e_{-\bar{\omega}} \rangle, \\
			C & \in & \langle e_0 + \lambda^{\s} e_1 \rangle,
		\end{array}
	\end{equation}
	which gives $1+4+4+1=10$ independent co\"{o}rdinates.  Notice that
	$X_2\notin U_2^{\perp}$, since $\lambda\neq\lambda^{\s}$.

	Since $\B_{\H}$ is non-degenerate, $(U_2^{\perp})^{\perp}=U_2$, and hence the radical
	of $\B_{\H}$ on $U_2^{\perp}$ is $U_2\cap U_2^{\perp}$. Comparing
	(\ref{eq:U2}) and (\ref{eq:U2perp}), and using $\lambda\neq\lambda^{\s}$, gives
	\begin{equation}
		U_2\cap U_2^{\perp}=R_2.
	\end{equation}
	Indeed, the intersections of the two $A$-spaces and of the two $B$-spaces are
	respectively $\langle e_{-0},e_1\rangle$ and
	$\langle e_{-\omega},e_{-\bar{\omega}}\rangle$, while
	$e_0+\lambda^{\s}e_1=(e_0+\lambda e_1)+(\lambda^{\s}-\lambda)e_1$ belongs to the
	$C$-space in (\ref{eq:U2}).

	Put $d=\lambda-\lambda^{\s}$. Write an arbitrary vector in $U_2^{\perp}$ in the
	form $(0,0,c\mid A,B,C)$, where
	\begin{equation*}
		\begin{array}{r@{\;}c@{\;}l}
			A&=&\alpha_1(e_{\bar{\omega}}+\lambda^{\s}e_{-\omega})+
			\alpha_2(e_{\omega}-\lambda^{\s}e_{-\bar{\omega}})+
			\alpha_3e_{-0}+\alpha_4e_1,\\
			B&=&\beta_1(e_{-1}+\lambda^{\s}e_0)+
			\beta_2(e_{-0}-\lambda^{\s}e_1)+
			\beta_3e_{-\omega}+\beta_4e_{-\bar{\omega}},\\
			C&=&\gamma(e_0+\lambda^{\s}e_1).
		\end{array}
	\end{equation*}
	A direct calculation gives
	\begin{equation}
		\label{eq:7_H_on_U2perp}
		\H((0,0,c\mid A,B,C))=cc^{\s}
		+d(\alpha_2\alpha_1^{\s}-\alpha_1\alpha_2^{\s})
		+d(\beta_2\beta_1^{\s}-\beta_1\beta_2^{\s}).
	\end{equation}
	Thus $\B_{\H}$ induces a non-degenerate Hermitean form on the $5$-space
	\begin{equation}
		\ovV=U_2^{\perp}/R_2.
	\end{equation}
	It is the orthogonal sum of the line with form $cc^{\s}$ and two hyperbolic
	Hermitean planes. In particular, its special unitary group is $\SU_{5,K}(F)$.
	This calculation is also valid in characteristic $2$, since $d\neq0$ and
	$d^{\s}=-d$.

	We next put $X_2$ and $U_2^{\perp}$ in the co\"{o}rdinates of
	\cite[Lemma~11.5]{BSW1}.  Since $1-\lambda^{\s}e_1$ has norm $1$, the
	element $P_{1-\lambda^{\s}e_1}$ belongs to $\SE_6(K)$.  Its action
	$(A,B,C)\mapsto(uA,Bu,\ovu C\ovu)$, with $u=1-\lambda^{\s}e_1$, sends
	the displayed basis of $U_1^{\perp}$ to that of $U_2^{\perp}$.
	Thus the calculation preceding Theorem~\ref{theorem:7_type1_stabiliser}
	gives $\bigl(\J_{10}^{abC}\bigr)^k=U_2^{\perp}$, where
	\begin{equation}
		\label{eq:7_type2_k}
		k=M M_{e_{-0}}P_{1-\lambda^{\s}e_1}.
	\end{equation}
	Set
	\begin{equation*}
		u_1=-(e_{-1}+e_{-0}+e_1),\qquad
		u_2=-e_{-1}+e_0+e_{-0},\qquad
		D_0=\diag(-d,1,-d^{-1}).
	\end{equation*}
	Both $u_1$ and $u_2$ have norm $1$.  The element
	\begin{equation}
		\label{eq:7_type2_h}
		h=M''_{d^{-1}e_{-1}}\delta P'_{u_1}P'_{u_2}D_0
	\end{equation}
	stabilises $\J_{10}^{abC}$.  Put
	\begin{equation}
		\label{eq:7_type2_kappa}
		\kappa=h^{-1}k,
	\end{equation}
	and write $p=(0,0,0\mid e_0,0,0)$ and $v=(0,0,1\mid0,0,0)$.
	Direct substitution gives
	\begin{equation*}
		\bigl(\J_{10}^{abC}\bigr)^\kappa=U_2^{\perp},\qquad
		p^\kappa=X_2,
	\end{equation*}
	and, for $X=(a,b,c\mid A,B,C)$,
	\begin{equation}
		\label{eq:7_type2_kappa_pairing}
		\B_{\H}(X_2,X^\kappa)=d^2c^{\s}.
	\end{equation}
	For $g\in G_2$, put $\widetilde g=\kappa g\kappa^{-1}$ and write
	$X^{\widetilde g}=(a',b',c'\mid A',B',C')$.  Preservation of
	$\B_{\H}$ and the fact that $g$ fixes $X_2$ give
	\begin{equation*}
		d^2(c')^{\s}=\B_{\H}(X_2,(X^{\widetilde g})^\kappa)
		=\B_{\H}(X_2,(X^\kappa)^g)
		=\B_{\H}(X_2,X^\kappa)=d^2c^{\s}.
	\end{equation*}
	Since $d\neq0$ and $\s$ is an automorphism, $c'=c$.  Thus
	$\widetilde g$ fixes $p$ and leaves the third scalar co\"{o}rdinate unchanged.

	By \cite[Lemma~11.5]{BSW1}, the joint stabiliser of $\J_{10}^{abC}$ and
	$\langle p\rangle$ has shape
	\begin{equation}
		K^{11}\cn K^{10}\cn\SL_5(K).K^{\times}.K^{\times}.
	\end{equation}
	Write $\widetilde Q$ for its normal subgroup $K^{11}\cn K^{10}$, and put
	$Q=\kappa^{-1}\widetilde Q\kappa$.  In the basis
	$v_1,\ldots,v_5,w_1,\ldots,w_5$ used in that proof, $\widetilde Q$ fixes
	$R=\langle v_1,\ldots,v_5\rangle$ pointwise and acts trivially on
	$\J_{10}^{abC}/R$, while the displayed $\SL_5(K)$ acts as
	\begin{equation*}
		\begin{bmatrix}S&0\\0&(S^{-1})^{\T}\end{bmatrix},\qquad S\in\SL_5(K).
	\end{equation*}
	Here $R^\kappa=R_2$.  The action on the quotient is faithful, since
	$(S^{-1})^{\T}=\II_5$ implies $S=\II_5$.
	We next show that fixing $p$ and the third scalar co\"{o}rdinate leaves
	precisely $\widetilde Q\cn\SL_5(K)$.
	Let $E\cong K^{16}$ be the normal subgroup in the stabiliser of
	$\J_{10}^{abC}$ from \cite[Theorem~11.2]{BSW1}.  Its proof shows that $E$
	acts sharply transitively on white vectors whose third scalar co\"{o}rdinate
	is $1$.  Choose $t\in E$ with $v^t=v^{\widetilde g}$.  Then
	$s=\widetilde g t^{-1}$ fixes $v$ and stabilises $\J_{10}^{abC}$, so it
	belongs to the displayed $\Spin_{10}^{+}(K)$.  This group preserves the
	decomposition $\J=\J_1^c\oplus\J_{16}^{AB}\oplus\J_{10}^{abC}$, whereas
	on $\J_{16}^{AB}$ an element of $E$ only adds a vector of
	$\J_{10}^{abC}$.  Since $p^{st}=p$, comparison of these components gives
	$p^s=p$ and then $p^t=p$.  Thus $t$ belongs to the $K^{11}$ in
	\cite[Lemma~11.5]{BSW1}.

	The element $s$ preserves the first $5$-space in the basis used there.
	Multiplying $s$ by elements of the displayed $K^{10}$ cancels its
	lower block.  The displayed $\SL_5(K)$ then reduces its first
	diagonal block to $\diag(\eta,1,1,1,1)$, for some $\eta\in K^{\times}$.
	All these elements fix $p$ and $v$, and preserve the above decomposition.
	The remaining element fixes the $C$-co\"{o}rdinate and scales $a$ and $b$
	inversely.  The calculation in the proof of \cite[Lemma~9.5]{BSW1} shows
	that its action on $\J$ is
	\begin{equation*}
		(a,b,c\mid A,B,C)\longmapsto
		(\mu^2a,\mu^{-2}b,c\mid\mu^{-1}A,\mu B,C)
	\end{equation*}
	for some $\mu\in K^{\times}$.  As it fixes $p$, we have $\mu=1$, so the
	remaining element is the identity.  Conversely, every element of the
	displayed $K^{11}$, $K^{10}$ and $\SL_5(K)$ fixes $p$ and the third scalar
	co\"{o}rdinate.  Hence the common fixer is exactly
	$\widetilde Q\cn\SL_5(K)$, and no additional scalar factor remains.

	Let
	\begin{equation}
		\rho:G_2\longrightarrow\GL(\ovV)
	\end{equation}
	be the action on $\ovV=U_2^{\perp}/R_2$.  The preceding calculation and
	(\ref{eq:7_H_on_U2perp}) show that
	\begin{equation}
		\rho(G_2)\leqslant\SU(\ovV)=\SU_{5,K}(F),
	\end{equation}
	and $\ker\rho=Q\cap G_2$.  Put $Q_0=\ker\rho$.  We first construct a
	complement to this normal subgroup, and then determine $Q_0$.
	Write
	\begin{equation*}
		\mathbf x=(c,\alpha_1,\alpha_2,\beta_1,\beta_2)^{\T}.
	\end{equation*}
	The form in
	(\ref{eq:7_H_on_U2perp}) has Gram matrix
	\begin{equation}
		\label{eq:7_type2_Gram}
		G=\begin{bmatrix}
		1&0&0&0&0\\
		0&0&d&0&0\\
		0&-d&0&0&0\\
		0&0&0&0&d\\
		0&0&0&-d&0
		\end{bmatrix},
	\end{equation}
	so that $\H(\mathbf x)=\mathbf x^{\s\T}G\mathbf x$.

	Put $\widetilde Q_0=\kappa Q_0\kappa^{-1}$.  On
	$\J_{10}^{abC}$, in the above standard basis, every element of
	$\widetilde Q$ has a lower block
	\begin{equation*}
		\begin{bmatrix}\II_5&0\\ A&\II_5\end{bmatrix},
	\end{equation*}
	where $A$ is alternating; in characteristic $2$ its diagonal is required
	separately to be zero.  These blocks add under multiplication.  The elements
	with zero lower block form the normal group $V\cong K^{11}$.  Its elements
	have the form $M''_xL'_y$, where
	\begin{equation*}
	\begin{split}
		x={}&x_{-1}e_{-1}+x_0e_0+x_{-\omega}e_{-\omega}
		+x_{-\bar\omega}e_{-\bar\omega},\\
		y={}&y_{-1}e_{-1}+y_{\bar\omega}e_{\bar\omega}+y_\omega e_\omega
		+y_{-0}e_{-0}+y_{-\omega}e_{-\omega}
		+y_{-\bar\omega}e_{-\bar\omega}+y_1e_1.
	\end{split}
	\end{equation*}
	On the standard co\"{o}rdinates put
	\begin{equation*}
		\B_\kappa(X,Y)=\B_{\H}(X^\kappa,Y^\kappa).
	\end{equation*}
	We construct the unitary subgroup explicitly.  Let $E_{ij}$ denote the
	$5\times5$ matrix unit.  For $p\in K$, direct substitution in
	(\ref{eq:5_nx_image}) gives the following elements fixing $X_2$, together
	with their matrices for the action $\mathbf x\mapsto T\mathbf x$ on $\ovV$:
	\begin{equation}
	\label{eq:7_type2_SU_generators}
	\begin{array}{c|c}
	N_{p e_{\bar\omega}}&\II_5-pE_{25}-p^{\s}E_{43}\\
	N_{p e_\omega}&\II_5-pE_{35}+p^{\s}E_{42}\\
	N_{p e_{-\omega}}&\II_5-pE_{24}+p^{\s}E_{53}\\
	N_{p e_{-\bar\omega}}&\II_5+pE_{34}+p^{\s}E_{52}\\
	N'_{p e_{\bar\omega}}N'_{p\lambda e_{-\omega}}&
	\II_5+pdE_{13}-p^{\s}E_{21}+\lambda^{\s}pp^{\s}E_{23}\\
	N'_{p e_\omega}N'_{-p\lambda e_{-\bar\omega}}&
	\II_5-pdE_{12}-p^{\s}E_{31}-\lambda^{\s}pp^{\s}E_{32}\\
	N''_{p e_{-1}}N''_{p\lambda^{\s}e_0}&
	\II_5-p^{\s}dE_{15}+pE_{41}+\lambda^{\s}pp^{\s}E_{45}\\
	N''_{p e_{-0}}N''_{-p\lambda^{\s}e_1}&
	\II_5+p^{\s}dE_{14}+pE_{51}-\lambda^{\s}pp^{\s}E_{54}.
	\end{array}
	\end{equation}
	Let $s_i(p)$ denote the element in row $i$, and put
	\begin{equation*}
		S_0=\langle s_i(p):1\leqslant i\leqslant8,\ p\in K\rangle.
	\end{equation*}
	Proposition~\ref{prop:5_nx_2se} and direct substitution for $X_2$
	give $S_0\leqslant G_2$.
	Each matrix $T$ in (\ref{eq:7_type2_SU_generators}) satisfies
	\begin{equation*}
		T^{\s\T}GT=G,\qquad \det(T)=1,
	\end{equation*}
	by direct multiplication.  The four matrices in the first four rows are the
	elementary operations joining the two hyperbolic planes.  The last four join
	the first co\"{o}rdinate to either hyperbolic plane.  To see the remaining
	elementary operations explicitly, let $U(p)$ denote the matrix in the fifth
	row.  Direct multiplication gives
	\begin{equation*}
		U(p)U(r)U(p+r)^{-1}
		=\II_5-\Tr_{K/F}(\lambda p^{\s}r)E_{23}.
	\end{equation*}
	The other three rows give, in the same way, $\II_5+sE_{32}$,
	$\II_5+sE_{45}$ and $\II_5+sE_{54}$ for every $s\in F$.  Indeed, the trace
	map is onto also in characteristic $2$, since $K/F$ is separable.

	Write $T_i(p)$ for the matrix in row $i$ of
	(\ref{eq:7_type2_SU_generators}).  We first show that these matrices act
	transitively on the isotropic $1$-spaces.  Let
	$\mathbf x=(c,\alpha_1,\alpha_2,\beta_1,\beta_2)^{\T}$ be nonzero and
	isotropic.  At least one of its last four co\"{o}rdinates is nonzero.  If
	$\alpha_2=0$, we make it nonzero by applying $\II_5+E_{32}$ when
	$\alpha_1\neq0$, $T_4(1)$ when $\beta_1\neq0$, or $T_2(1)$ when
	$\beta_2\neq0$.  Applying $T_5(-c/(d\alpha_2))$ then clears $c$.
	Next choose $p,r\in K$ with $p^{\s}=\beta_1/\alpha_2$ and
	$r^{\s}=-\beta_2/\alpha_2$; applying $T_1(p)$ and then $T_3(r)$ clears
	$\beta_1$ and $\beta_2$, without changing $c$ or $\alpha_2$.
	The resulting isotropic vector is $(0,\alpha_1,\alpha_2,0,0)^{\T}$, so
	$\alpha_1/\alpha_2\in F$.  The matrix
	$\II_5-(\alpha_1/\alpha_2)E_{23}$ therefore sends it to the
	$\alpha_2$-axis.  This proves the transitivity.

	The matrices $\II_5+sE_{32}$, $s\in F$, are precisely the unitary
	transvections with centre the $\alpha_2$-axis, together with the identity.
	Their conjugates therefore give all unitary transvections.  These generate
	the special unitary group of this $5$-space, by
	\cite[proof of Theorem~10.23]{Taylor}.  Hence the matrices in
	(\ref{eq:7_type2_SU_generators}) generate $\SU(G)$.

	The generation argument gives $\rho(S_0)=\SU_{5,K}(F)$.  To prove that
	this action of $S_0$ is faithful, we use the vector $v=(0,0,1\mid0,0,0)$
	in the standard co\"{o}rdinates.  Direct substitution gives
	\begin{equation*}
		v^\kappa=(0,0,0\mid0,0,
		d(-e_{-1}+\lambda^{\s}e_0+\lambda e_{-0}
		+\lambda\lambda^{\s}e_1)).
	\end{equation*}
	Substituting this vector in (\ref{eq:5_nx_image}) shows that every
	$s_i(p)$ fixes it.  Thus every element of $\kappa S_0\kappa^{-1}$ fixes $v$.

	Let $W$ be the subspace of $\J_{10}^{abC}$ orthogonal to $\J_{16}^{AB}$
	with respect to $\B_\kappa$.  In the standard basis from
	\cite[Lemma~11.5]{BSW1}, the pairing gives
	\begin{equation}
		\label{eq:7_type2_complement_space}
		W=\langle w_1,\ w_2+2v_3,\ w_3-2v_2,\ w_4-v_5,\ w_5+v_4\rangle,
		\qquad \J_{10}^{abC}=R\oplus W.
	\end{equation}
	Indeed, the five displayed vectors are orthogonal to $\J_{16}^{AB}$.
	The pairings of $v_1,\ldots,v_5$ with the co\"{o}rdinate vectors in the
	$A_{-0},B_0,B_{-1},B_{-\bar\omega},B_{-\omega}$ positions form the diagonal
	matrix $\diag(d^2,-d^{-1},-d^{-1},-d^{-1},d^{-1})$.
	It is nonsingular, so the five displayed vectors span the whole
	orthogonal subspace.  This remains valid in characteristic $2$.

	Suppose $q\in\widetilde Q_0$ fixes $v$.  As in the preceding
	joint-stabiliser calculation, $q$ belongs to the displayed
	$\Spin_{10}^{+}(K)$ and therefore preserves $\J_{16}^{AB}$.  Since it also
	preserves $\B_\kappa$, it preserves $W$.  For $w\in W$, the difference
	$w^q-w$ belongs to $R$, because $\widetilde Q$ acts trivially on
	$\J_{10}^{abC}/R$, and also belongs to $W$.  Thus $w^q=w$.
	As $q$ already fixes $R$ pointwise, it fixes $\J_{10}^{abC}$ pointwise
	and hence belongs to $V$.  The formula for $M''_xL'_y$ shows that fixing
	$v$ forces $x=y=0$.  Therefore $q=1$.
	It follows that $S_0\cap Q_0=1$.  For any $g\in G_2$, choose $s\in S_0$
	with the same action on $\ovV$.  Then $gs^{-1}\in Q_0$, so
	\begin{equation}
		\label{eq:7_type2_split}
		G_2=Q_0\rtimes S_0,\qquad S_0\cong\SU_{5,K}(F).
	\end{equation}

	We now determine $Q_0$.  Put $Z=V\cap\widetilde Q_0$, and let
	$t=M''_xL'_y\in Z$.  For $z=(0,0,0\mid A,B,0)$, the elementary action gives
	\begin{equation*}
		z^t-z=(\Tr(\ovB x),\Tr(Ay),0\mid0,0,
		\ovx\ovA+\ovB y)\in\J_{10}^{abC}.
	\end{equation*}
	Since $t$ fixes $\J_{10}^{abC}$ pointwise and preserves $\B_\kappa$,
	this difference is orthogonal to $\J_{10}^{abC}$, and hence belongs to
	its radical $R$.  Every vector of $R$ has first scalar co\"{o}rdinate zero,
	so $\Tr(\ovB x)=0$ for every $B$, giving $x=0$.  The $C$-co\"{o}rdinate
	then gives
	\begin{equation*}
		\OO_K y\subseteq
		\langle e_{-1},e_{-0},e_{-\omega},e_{-\bar\omega}\rangle.
	\end{equation*}
	Applying the multiplication table to $y$, $e_{-1}y$, and $e_{-\omega}y$
	shows successively that $y=s e_{-0}$.  Now $v^t=v+sp$, and
	(\ref{eq:7_type2_kappa_pairing}) gives
	\begin{equation*}
		0=\B_\kappa(v+sp,v+sp)-\B_\kappa(v,v)=d^2(s+s^{\s}).
	\end{equation*}
	Thus $s+s^{\s}=0$.

	To prove the converse and obtain the remaining parameters, put
	$B_{ij}=E_{ij}-E_{ji}$.  Direct substitution in
	(\ref{eq:5_nx_image}) shows that the following four families fix $X_2$ and act
	trivially on $\ovV$; after conjugating by $\kappa$, their lower blocks are as
	indicated:
	\begin{equation}
		\label{eq:7_type2_kernel_seeds}
	\begin{array}{c|c}
	\text{element of $Q_0$}&\text{lower block}\\ \hline
	N'_{\mu e_{-0}}&-\mu^{\s}d^{-1}B_{12}\\
	N'_{\mu e_1}&-\mu^{\s}d^{-1}B_{13}\\
	N''_{\mu e_{-\omega}}&-\mu d^{-1}B_{14}\\
	N''_{\mu e_{-\bar\omega}}&-\mu d^{-1}B_{15}.
	\end{array}
	\qquad(\mu\in K)
	\end{equation}
	The elements in the left-hand column belong to ${}^2\SE_{6,K}(F)$ by
	Proposition~\ref{prop:5_nx_2se}; hence they do belong to $Q_0$.
	With $[x,y]=x^{-1}y^{-1}xy$, direct calculation gives
	\begin{equation}
		\label{eq:7_type2_kernel_commutator}
		[\kappa N'_{\mu e_{-0}}\kappa^{-1},
		\kappa N'_{\nu e_1}\kappa^{-1}]
		=L'_{d\Tr_{K/F}(\mu\nu^{\s})e_{-0}}
		\qquad(\mu,\nu\in K).
	\end{equation}
	The trace map is onto since $K/F$ is separable, and its kernel is $dF$,
	also in characteristic $2$.  The right-hand side therefore supplies all
	the elements allowed by the preceding calculation.  Hence
	\begin{equation}
		\label{eq:7_type2_V0}
		Z=\{L'_{s e_{-0}}:s+s^{\s}=0\}\cong(F,+).
	\end{equation}
	Moreover, $Z$ is central in $\widetilde Q$.  Indeed, its elements lie in
	the abelian group $E$, so they commute with $V$.  The displayed $K^{10}$
	normalises $E$ and fixes $v$ and $p$.  Conjugation by this $K^{10}$ therefore
	preserves the unique element of $E$ sending $v$ to $v+sp$, namely
	$L'_{s e_{-0}}$.  Thus $Z$ also commutes with $K^{10}$, as required.

	It remains to obtain every alternating lower block.  Let $\varepsilon_i$
	be the $i$-th standard row vector of $K^5$.  Since lower blocks add under
	multiplication, the four families in (\ref{eq:7_type2_kernel_seeds})
	supply every block
	\begin{equation*}
		\varepsilon_1^{\T}\mathbf z-\mathbf z^{\T}\varepsilon_1,
		\qquad\mathbf z\in K^5.
	\end{equation*}
	For $i=5,6,7$, let $S_i$ be the upper-left block of
	$\kappa s_i(1)\kappa^{-1}$ in the standard basis
	$v_1,\ldots,v_5,w_1,\ldots,w_5$.  Direct substitution gives
	\begin{equation*}
		\varepsilon_1S_5=\varepsilon_1+d^2\varepsilon_4,\qquad
		\varepsilon_1S_6=\varepsilon_1-d^2\varepsilon_5,\qquad
		\varepsilon_1S_7=\varepsilon_1-d^2\varepsilon_2.
	\end{equation*}
	Normality of $Q_0$ allows us to conjugate by these elements.  On lower
	blocks this sends $A$ to $S_i^{\T}AS_i$.  Since $\mathbf zS_i$ ranges over
	all row vectors, subtracting the known first-row-and-column blocks gives
	all blocks $\varepsilon_j^{\T}\mathbf z-\mathbf z^{\T}\varepsilon_j$
	for $j=4,5,2$; here we use only $d\neq0$.  Together these supply all ten
	independent entries of an alternating $5\times5$ matrix.  Two elements of
	$\widetilde Q_0$ have the same lower block precisely when they differ by
	an element of $Z$.  Therefore
	\begin{equation}
		\label{eq:7_type2_kernel}
		\widetilde Q_0/Z\cong\operatorname{Alt}_5(K)\cong(K,+)^{10}.
	\end{equation}
	These arguments remain valid in characteristic $2$.
	The commutator in (\ref{eq:7_type2_kernel_commutator}) shows that
	$\widetilde Q_0$ is non-abelian.  If $Z$ had a complement in
	$\widetilde Q_0$, that complement would be abelian by
	(\ref{eq:7_type2_kernel}).  Since $Z$ is central, $\widetilde Q_0$ would
	then be abelian, a contradiction.  Thus $Q_0$ has non-split shape
	$F.K^{10}$.  Together with (\ref{eq:7_type2_split}), this proves the
	required shape $(F.K^{10})\cn\SU_{5,K}(F)$.

\end{proof}

\section{Simplicity of ${}^2\E_{6,K}(F)$}
\label{section:simplicity}

We retain hypothesis~\eqref{eq:1_norm_surjective}. As in
\cite[Section~10]{BSW1}, we use Iwasawa's lemma.

\begin{lemma}[Iwasawa]
    A perfect group acting faithfully and primitively on a set is simple
    if a point stabiliser has an abelian normal subgroup whose conjugates
    generate the group.
\end{lemma}

Put $G={}^2\SE_{6,K}(F)$, and let $\mathcal P$ be its orbit of type-$1$
white points. Write
\[
    x=X_1=(0,0,0\mid0,0,e_0),\qquad
    x^*=(0,0,0\mid0,0,e_{-0}),\qquad p=X_3,
\]
and put $P=G_{\langle x\rangle}$ and $S=G_p\cong\Spin^-_{10,K}(F)$.
We use $H\lhd G_x$ and $Z=[H,H]\cong(F,+)^8$ from the
type-$1$ stabiliser calculation.

We first pass from the vector stabiliser $G_x$ to $P$. For $t\in K^\times$, set
\begin{equation}
    \label{eq:simplicity_diagonal}
    u_t=t e_0+(t^{\s})^{-1}e_{-0},\qquad
    D_t=\diag(1,u_t,u_t^{-1}).
\end{equation}
The algebra $Ke_0+Ke_{-0}$ is sociable. The factorisation
$D_t=M'_{u_t-1}L'_1M'_{u_t^{-1}-1}L'_{-u_t}$, in the notation of
\cite[Section~6]{BSW1}, shows that $D_t$ preserves $\fdet$. Its action is
\[
    (a,b,c\mid A,B,C)^{D_t}
    =\left(a,\frac{t}{t^{\s}}b,\frac{t^{\s}}t c\ \middle|\
      \bar u_t A u_t^{-1},\overline{u_t^{-1}}B,Cu_t\right).
\]
The scalar multipliers have norm $1$, and opposite octonion
co\"ordinates have multipliers $r$ and $(r^{\s})^{-1}$.
Thus $D_t$ also preserves $\B_{\H}$, and
\[
    x^{D_t}=tx,\qquad p^{D_t}=(t^{\s}/t)p,\qquad
    P=G_x\{D_t:t\in K^\times\}.
\]
Conjugation by $D_t$ rescales the eight families generating $H$.
For $D_t^{-1}N'_{\mu e_i}D_t$ the new parameter is
$\mu t^{\s}/t$ when $i=1,\omega,\ombb$, and $\mu(t^{\s})^2$ when
$i=-0$. For $D_t^{-1}N''_{\nu e_i}D_t$ it is $\nu t$ when
$i=-1,-\omega,-\ombb$, and $\nu/t^{\s}$ when $i=-0$.
Hence $H\lhd P$, and consequently $Z\lhd P$.

\begin{lemma}
    \label{lemma:simplicity_generation}
    The group $G$ is perfect and is generated by the conjugates of $Z$.
    Moreover, $\langle P,S\rangle=G$.
\end{lemma}

\begin{proof}
    Put $E=\langle S,S^\tau,S^{\tau^2}\rangle$.
    These three vector stabilisers contain all the transformations used
    in the proof of Theorem~\ref{theorem:6_three_orbits}, including the
    interchanges of matrix co\"ordinates. Thus $E$ is transitive on
    type-$3$ points. Each $D_t$ fixes the first scalar basis vector,
    so belongs to $E$, and $t^\s/t$ runs through all norm-$1$ scalars
    \cite[Lemma~10.1(iv)]{Taylor}. Hence $E$ is transitive on white
    vectors of Hermitean value $1$. Since $G_p=S\leqslant E$, we have $E=G$.

    The natural image $\Omega^-_{10,K}(F)$ of $S$ is perfect
    \cite[Theorem~11.47]{Taylor}. Its kernel is trivial in characteristic
    $2$; otherwise it lies in the subgroup
    $\Spin^+_6(F)\cong\SL_4(F)$ supplied by three hyperbolic pairs,
    which is perfect \cite[Theorem~4.4]{Taylor}. Thus $S$ is perfect.
    The $S$-conjugates of $Z$ induce all the groups of Siegel
    transformations, which generate the natural image
    \cite[Theorem~11.46]{Taylor}. If $J$ is the subgroup they generate
    in $S$, then $S=J\{P_1,P_{-1}\}$. The kernel is central, so
    $S=[S,S]\leqslant J$. The three cyclic copies of $S$ therefore
    show both that $G$ is perfect and that the conjugates of $Z$ generate $G$.

    Finally, over the sociable algebra $Ke_0+Ke_{-0}$, the encoding
    \[
        T\longmapsto e_0T+e_{-0}(T^\s)^{-\T}
        \qquad(T\in\SL_3(K))
    \]
    is multiplicative. The families $N_{\mu e_0},N_{\mu e_{-0}}$
    give the elementary $E_{12}$- and $E_{21}$-families, while
    $N'_{\mu e_{-0}},N''_{\mu e_{-0}}$ give the $E_{32}$- and
    $E_{13}$-families. The first two lie in $S$, and the last two in $P$.
    Their commutators supply the remaining elementary matrices.
    Hence $\langle P,S\rangle$ contains $\tau$, and therefore all
    three cyclic copies of $S$. Thus $\langle P,S\rangle=G$.
\end{proof}

For the primitivity argument, recall
\[
    A_4=\langle e_{-0},e_1,e_\omega,e_{\ombb}\rangle_K,\qquad
    B_4=\langle e_{-0},e_{-1},e_{-\omega},e_{-\ombb}\rangle_K,
\]
and the radical
\[
    R=\{(0,0,0\mid A,B,C):A\in A_4,\ B\in B_4,\ C\in Ke_0\}.
\]
Subscripts on $A,B,C$ below denote their coefficients in the octonion basis.

\begin{lemma}
    \label{lemma:simplicity_primitive}
    The action of $G$ on $\mathcal P$ is primitive.
\end{lemma}

\begin{proof}
    First, $H$ is transitive on the points of $\mathcal P$ not orthogonal
    to $x$. Indeed, normalise $C_{-0}=1$. Formula~\eqref{eq:5_nx_image}
    shows that the eight families generating $H$ independently clear
    \[
        B_0,B_1,B_\omega,B_{\ombb},
        A_0,A_{-1},A_{-\omega},A_{-\ombb},
    \]
    without changing $C_{-0}$. We now have $A\in A_4$, $B\in B_4$.
    The $A_4$-component of $CA=b\bar B$ gives $A=0$, and the
    $B_4$-component of $BC=a\bar A$ gives $B=0$. Then $AB=c\bar C$
    gives $c=0$. By the proof of Theorem~\ref{theorem:6_three_orbits}, our normalised
    type-$1$ vector belongs to
    \[
        V=\{(a,-a^\s,0\mid0,0,C):a\in K,\ C\in\OO_F\}.
    \]
    On $V$, put $Q=ab-\NN(C)$ and let $f$ be its polar form.
    The group $Z$ is transitive on singular vectors with $C_{-0}=1$,
    by the matrices in \cite[Lemma~2.1]{BSW1}. Thus it sends our vector
    to $x^*$, proving the assertion.

    Let $\mathcal D$ be a block containing $\langle x\rangle$, so that
    $P$ preserves $\mathcal D$.
    If it contains two non-orthogonal points, move the first to
    $\langle x\rangle$. The image block meets $\mathcal D$, so is
    $\mathcal D$ itself. The preceding transitivity then puts
    $\langle x^*\rangle$ in $\mathcal D$.
    Now $S$ acts primitively on the singular $F$-points of $V$
    \cite[Theorem~11.30]{Taylor}. The block equivalence relation
    identifies the distinct points $\langle x\rangle,\langle x^*\rangle$,
    so identifies that entire $S$-orbit. Hence $S$ preserves $\mathcal D$.
    Since $P$ also preserves it and $\langle P,S\rangle=G$,
    we obtain $\mathcal D=\mathcal P$.

    Suppose instead that all points of $\mathcal D$ are mutually orthogonal.
    We show directly that their representatives lie in $R$.
    Let $Y=(a,b,c\mid A,B,C)$ represent a point of $\mathcal D$.
    For a basis octonion $e_j\in A_4$, put $d=\Tr(A\overline{e_j})$.
    Since $N'_{\mu e_j}\in P$, orthogonality and
    \eqref{eq:5_nx_image} give, for some $\alpha,\beta\in K$ independent of $\mu$,
    \[
        0=\B_{\H}(Y,Y^{N'_{\mu e_j}})
          =\mu\alpha+\mu^\s\beta-\mu\mu^\s dd^\s.
    \]
    Here the quadratic term follows from $e_j\bar A e_j=d e_j$.
    Subtracting the instances for $\mu$ and $\nu$ from that for
    $\mu+\nu$ gives
    \[
        \Tr_{K/F}(\mu\nu^\s)\,dd^\s=0.
    \]
    The trace map is onto, so $d=0$. This gives $A\in A_4$.
    The identical calculation with $N''_{\mu e_j}$, for $e_j\in B_4$,
    gives $B\in B_4$. Using $N_{\mu e_j}\in P$ for $j\neq\pm0$
    gives $C_j=0$ for all $j\neq\pm0$.
    Also $C_{-0}=0$, since $Y$ is orthogonal to $x$.
    Applying $N_{\mu e_j}$ once more, its new $C_j$-coefficient
    is $a\mu-b\mu^\s$, which must vanish for every $\mu$.
    Taking $\mu=1$ and then $\mu\notin F$ gives $a=b=0$.
    Finally $\H(Y)=cc^\s=0$ gives $c=0$, so $Y\in R$.

    Put $U=\langle\mathcal D\rangle_K\subseteq R$.
    Since $P$ preserves $\mathcal D$, the space $U$ is $P$-invariant.
    On $R$, the whiteness equations reduce to
    \[
        q(A,B):=A_{-0}B_{-0}-A_1B_{-1}
                 -A_\omega B_{-\omega}-A_{\ombb}B_{-\ombb}=0.
    \]
    Adding any multiple of $x$ to a white vector in $U$ gives a white
    vector or zero. The block stabiliser preserves $U$ and acts
    transitively on $\mathcal D$, so the same assertion holds with
    any representative of a point of $\mathcal D$ in place of $x$.
    Thus the polar form of $q$ vanishes on every pair of such
    representatives. Since they span $U$, we have $q|_U=0$.

    Omitting the $Ke_0$-co\"ordinate identifies $U/Kx$ with a subspace
    of $A_4\oplus B_4$. On this space, $N_{\mu e_j}$, for $j\neq\pm0$,
    acts as $\II+\mu r_j+\mu^\s t_j$, where
    \[
        r_j(A,B)=(\overline{e_j}\bar B,0),\qquad
        t_j(A,B)=(0,-\bar A\overline{e_j}).
    \]
    Taking $\mu=1$ and $\mu\notin F$ shows that $U/Kx$ is invariant
    under $r_j,t_j$ separately. The multiplication table gives
    \[
    \begin{array}{c|c}
        \text{map}&\text{image of }(A,B)\\ \hline
        t_{-\ombb}\circ r_{-\omega}\circ t_{-1}
            &(0,-A_{-0}e_{-0})\\
        t_{\ombb}\circ r_\omega\circ t_{-1}
            &(0,A_1e_{-1})\\
        r_{\ombb}\circ t_\omega\circ r_1
            &(-B_{-0}e_{-0},0)\\
        r_{-\ombb}\circ t_{-\omega}\circ r_1
            &(B_{-1}e_1,0).
    \end{array}
    \]
    Pairing $(A,B)$ with these images under the polar form of $q$
    gives respectively $-A_{-0}^2,-A_1^2,-B_{-0}^2,-B_{-1}^2$.
    They must all be zero. Cycling $1,\omega,\ombb$ gives the other
    four coefficients, so $A=B=0$. Hence $U=Kx$, and $\mathcal D$
    is a singleton. All these calculations remain valid in characteristic $2$.
\end{proof}

\begin{lemma}
    \label{lemma:simplicity_kernel}
    The kernel of the action of $G$ on $\mathcal P$ is the scalar subgroup
    \[
        Z_{\mathrm{sc}}=\{\zeta\II_{27}:\zeta^3=1,\ \zeta\zeta^\s=1\}.
    \]
\end{lemma}

\begin{proof}
    Let $g$ fix every point of $\mathcal P$.
    In $V$ choose two orthogonal hyperbolic pairs $u_1,v_1$ and $u_2,v_2$,
    with $f(u_i,v_i)=1$. Their points are type $1$.
    The points represented by $u_1+u_2,u_1+v_2,v_1+u_2$ force a common
    scalar multiplier $\zeta$ on all four basis vectors.
    For $w$ orthogonal to both pairs, the singular vectors
    $u_i+w-Q(w)v_i$, for $i=1,2$, then give $w^g=\zeta w$.
    Thus $g$ is scalar on $V\otimes_F K=\J_{10}^{abC}$.
    The same argument on its two cyclic images gives scalars on
    $\J_{10}^{bcA}$ and $\J_{10}^{caB}$.
    Their non-zero intersections force the scalars to agree, and the
    three spaces span $\J$. Hence $g=\zeta\II_{27}$.
    Conversely, such a scalar fixes every point, and belongs to $G$
    precisely when $\zeta^3=1$ and $\zeta\zeta^\s=1$.
\end{proof}

\begin{theorem}
    \label{theorem:simplicity}
    Under hypothesis~\eqref{eq:1_norm_surjective}, the group
    ${}^2\E_{6,K}(F)$ is simple.
\end{theorem}

\begin{proof}
    The preceding lemmas show that $G/Z_{\mathrm{sc}}$ is perfect and
    acts faithfully and primitively on $\mathcal P$.
    The image of $Z\lhd P$ is abelian, and its conjugates generate this
    quotient. Iwasawa's lemma therefore makes $G/Z_{\mathrm{sc}}$ simple.
    Being simple, non-trivial and perfect, this quotient has trivial centre.
    Since $Z_{\mathrm{sc}}\leqslant Z(G)$, it follows that
    \begin{equation}
        \label{eq:simplicity_centre}
        Z(G)=Z_{\mathrm{sc}}.
    \end{equation}
    Thus the simple quotient is precisely ${}^2\E_{6,K}(F)$.
\end{proof}

\section{The case of a finite field}

Let $F=\Fq$, $K=\Fqs$, and $\s:t\mapsto t^q$, where $q$ is a prime
power.  The norm map $t\mapsto t^{q+1}$ maps $K^{\times}$ onto
$F^{\times}$, so hypothesis~\eqref{eq:1_norm_surjective} holds.
Write
\[
    G={}^2\SE_6(q)={}^2\SE_{6,K}(F),\qquad
    {}^2\E_6(q)=G/Z(G).
\]
Let $o_i$ be the length of the orbit of $\langle X_i\rangle$, and let
$s_i$ be the order of its stabiliser in $G$, for $i=1,2,3$.  Here we use
the representatives from Section~7, in particular
$X_3=(0,0,1\mid0,0,0)$.  By Theorem~\ref{theorem:6_three_orbits}, these
are all the white-point orbits.  Substituting $q^2$ for $q$ in
\cite[Corollary~8.2]{BSW1}, the total number of white points is
\begin{equation}
    \label{eq:8_white_points}
    o_1+o_2+o_3=N,\qquad
    N=\frac{(q^{24}-1)(q^{18}-1)}{(q^8-1)(q^2-1)}.
\end{equation}

We first pass from vector stabilisers to point stabilisers.  The elements
$D_t$ of~\eqref{eq:simplicity_diagonal} now have
\[
    u_t=t e_0+t^{-q}e_{-0},\qquad
    D_t=\diag(1,u_t,u_t^{-1}).
\]
Their action, already calculated in Section~\ref{section:simplicity}, gives
\begin{equation}
    \label{eq:8_point_multipliers}
    X_1^{D_t}=tX_1,\qquad
    X_2^{D_t}=tX_2,\qquad
    X_3^{D_t}=t^{q-1}X_3.
\end{equation}
Consequently, for $i=1,2$, the orbit of $X_i$ under its point stabiliser
consists of all $q^2-1$ non-zero vectors on $\langle X_i\rangle$.
For $i=3$, preservation of $\H(X_3)=1$ permits precisely the multipliers
of norm $1$, and all $q+1$ of these occur by
\eqref{eq:8_point_multipliers}.  The point-stabiliser orders are
therefore $q^2-1$, $q^2-1$, and $q+1$ times the respective
vector-stabiliser orders.

The classical group order formulae (see \cite[pp.~118, 165]{Taylor}) give
\begin{align*}
    |\Spin^-_{2n}(q)|
        &=q^{n(n-1)}(q^n+1)\prod_{j=1}^{n-1}(q^{2j}-1)
          \qquad(n=4,5),\\
    |\SU_5(q)|
        &=q^{10}(q^2-1)(q^3+1)(q^4-1)(q^5+1).
\end{align*}
For odd $q$, the kernel of the natural action of the spin group has
order $2$, which cancels the index $2$ of $\Omega^-_{2n}(q)$ in
$\SO^-_{2n}(q)$; for even $q$ both factors are $1$.  Thus no extra
factor is needed in the displayed spin-group order.  Theorems
\ref{theorem:7_type1_stabiliser}, \ref{theorem:7_type2_stabiliser}, and
\ref{theorem:7_type3_stabiliser} now give
\begin{equation}
    \label{eq:8_point_stabiliser_orders}
    \begin{aligned}
        s_1&=(q^2-1)q^{24}|\Spin^-_8(q)|
             =q^{36}(q^2-1)^2(q^6-1)(q^8-1),\\
        s_2&=(q^2-1)q^{21}|\SU_5(q)|
             =q^{31}(q^2-1)^2(q^3+1)(q^4-1)(q^5+1),\\
        s_3&=(q+1)|\Spin^-_{10}(q)|\\
             &\phantom{{}={}}=q^{20}(q+1)(q^2-1)(q^4-1)(q^6-1)(q^8-1)(q^5+1).
    \end{aligned}
\end{equation}

By the orbit--stabiliser formula,
\[
    o_1s_1=o_2s_2=o_3s_3=|G|.
\]
Together with \eqref{eq:8_white_points}, this gives
\begin{equation}
    \label{eq:8_order_from_orbits}
    |G|\left(\frac1{s_1}+\frac1{s_2}+\frac1{s_3}\right)=N,
    \qquad
    o_i=\frac{N}{s_i(1/s_1+1/s_2+1/s_3)}.
\end{equation}
Solving these equations gives
\begin{equation}
    \label{eq:8_orbit_lengths}
    \begin{aligned}
        o_1&=\frac{(q^{12}-1)(q^9+1)(q^5+1)}{q^2-1},\\
        o_2&=\frac{q^5(q^{12}-1)(q^9+1)(q^4+1)(q^3-1)}{q^2-1},\\
        o_3&=\frac{q^{16}(q^{12}-1)(q^9+1)}{(q^4-1)(q+1)}.
    \end{aligned}
\end{equation}

Finally, Lemma~\ref{lemma:simplicity_kernel} and
\eqref{eq:simplicity_centre} give
\[
    Z(G)=\{\zeta\II_{27}:\zeta^3=\zeta^{q+1}=1\},\qquad
    |Z(G)|=\gcd(3,q+1).
\]
Using $|G|=o_1s_1$ and $|{}^2\E_6(q)|=|G|/|Z(G)|$, we have proved:

\begin{theorem}
    \label{theorem:8_finite_orders}
    For every prime power $q$,
    \begin{align*}
        |{}^2\SE_6(q)|
            &=q^{36}(q^2-1)(q^5+1)(q^6-1)(q^8-1)(q^9+1)(q^{12}-1),\\
        |{}^2\E_6(q)|
            &=\frac{q^{36}}{\gcd(3,q+1)}
              (q^2-1)(q^5+1)(q^6-1)(q^8-1)(q^9+1)(q^{12}-1).
    \end{align*}
\end{theorem}

\section{Conclusions}
\label{section:conclusions}

Hypothesis~\eqref{eq:1_norm_surjective} is essential to the three-orbit
statement: it is what allows every non-zero Hermitean value to be
normalised to $1$. We now drop this hypothesis, retaining only that
$K/F$ is a quadratic Galois extension. The same orbit arguments give
the following more general result.

\begin{theorem}
    \label{theorem:conclusions_norm_orbits}
    For an arbitrary quadratic Galois extension $K/F$, the group
    $G={}^2\SE_{6,K}(F)$ has precisely two orbits on isotropic white
    points, as in Theorem~\ref{theorem:6_three_orbits}. Its orbits on
    non-isotropic white points are in one-to-one correspondence with
    \[
        F^\times/\NN_{K/F}(K^\times).
    \]
    The orbit of $\langle X\rangle$ corresponds to the class of $\H(X)$.
\end{theorem}

\begin{proof}
    Since $G$ preserves $\H$ and $\H(tX)=tt^\s\H(X)$, the class of
    $\H(X)$ is independent of the non-zero representative of the point
    and is constant on each non-isotropic orbit.

    Conversely, suppose that $\H(X)$ and $\H(Y)$ lie in the same
    non-zero norm class. The reduction in
    Lemma~\ref{lemma:6_reduce_J10} uses no norm-surjectivity assumption,
    so we may move both points into $\J_{10}^{abC}$.
    Rescaling a representative changes its Hermitean value by a norm,
    so the two reduced points still have values in the same norm class.
    We may therefore rescale one representative to arrange
    $\H(X)=\H(Y)$. The proof of
    Lemma~\ref{lemma:3_2forms_nonisotropic}, after its initial rescaling,
    now supplies the required Witt isometry: equality of the Hermitean
    values makes the two associated $F$-planes isometric.
    Neither this step nor the construction of $\Omega(U,Q)$ and the
    adjustment of the isometry in the proof of
    Theorem~\ref{theorem:6_three_orbits} uses
    \eqref{eq:1_norm_surjective}. Thus an element of $G$ sends
    $\langle X\rangle$ to $\langle Y\rangle$.
    The reduction and the two isotropic orbit arguments in that proof
    likewise do not use \eqref{eq:1_norm_surjective}, so those two
    orbits remain unchanged.

    Finally, every norm class occurs. Fix $\lambda\in K\setminus F$
    and put $d=\lambda-\lambda^\s$, so $d^2\in F^\times$.
    For each $r\in F^\times$, set
    \[
        C_r=e_0+\lambda e_1-\frac{r\lambda}{d^2}e_{-1}
                         +\frac{r\lambda^2}{d^2}e_{-0}.
    \]
    Direct substitution gives
    \[
        \NN(C_r)=0,\qquad
        h(C_r)=\frac{r}{d^2}
          \bigl(\lambda^2+(\lambda^\s)^2-2\lambda\lambda^\s\bigr)=r.
    \]
    Hence $(0,0,0\mid0,0,C_r)$ is a white vector of Hermitean value $r$.
    This calculation also holds in characteristic $2$.
\end{proof}

Thus there are exactly three white-point orbits if and only if
\eqref{eq:1_norm_surjective} holds. If it fails, there are at least two
non-isotropic orbits, and there are exactly two precisely when
\[
    [F^\times:\NN_{K/F}(K^\times)]=2.
\]
For example, let $F=\mathbb R$, $K=\mathbb C$, and let $\s$ be complex
conjugation. The non-zero norms are precisely the positive real numbers.
There are therefore two non-isotropic orbits, distinguished by the sign
of the Hermitean value. Representatives are
\[
    \left\langle(0,0,1\mid0,0,0)\right\rangle,\qquad
    \left\langle(0,0,0\mid0,0,e_0+i e_1-i e_{-1}-e_{-0})\right\rangle,
\]
whose displayed vectors have Hermitean values $1$ and $-4$,
respectively. Together with the two isotropic orbits, this gives
four white-point orbits.

Failure of \eqref{eq:1_norm_surjective} alone does not imply that
there are exactly two non-isotropic orbits. For instance, for
$K=\mathbb Q(i)$ and $F=\mathbb Q$, the classes of $1,-1,3$ are
distinct. Every non-zero norm is positive, and $3$ is not a norm:
clearing denominators in $a^2+b^2=3$ would give a primitive integer
solution of $m^2+n^2=3k^2$, whereas reduction modulo $3$ forces
$3$ to divide $m,n$, and then $k$. The theorem therefore gives at
least three non-isotropic orbits in this example.
The stabiliser and simplicity assertions proved under
\eqref{eq:1_norm_surjective} are not being extended here.

\end{document}